\documentclass{compositio}
\usepackage{amsmath, amssymb, latexsym, graphicx, mathrsfs, enumerate}

\numberwithin{equation}{section}
\theoremstyle{plain}
\newtheorem{Thm}[equation]{Theorem}
\newtheorem{Prop}[equation]{Proposition}

\newtheorem{Lem}[equation]{Lemma}

\theoremstyle{definition}
\newtheorem{Def}[equation]{Definition}

\theoremstyle{remark}
\newtheorem{Rmk}[equation]{Remark}
\newtheorem{Exa}[equation]{Example}

\begin{document}

\title[Group schemes and local densities of quadratic lattices when $p=2$]
{Group schemes and local densities of quadratic lattices in residue characteristic 2}

\author{Sungmun Cho}
\email{sungmuncho12@gmail.com}
\address{Department of Mathematics\\Purdue University\\West Lafayette\\IN\\USA\\47907}
\curraddr{Department of
Mathematics\\University of Toronto\\Toronto\\ON\\CANADA\\M5S 2E4}

\dedication{Dedicated to my parents}



\begin{abstract}
The celebrated Smith-Minkowski-Siegel mass formula expresses the  mass of a quadratic lattice $(L, Q)$ as a product of local factors, called the local densities of $(L,Q)$. This mass formula is an essential tool for the classification of integral quadratic lattices.
In this paper, we will describe the local density formula  explicitly by observing the existence of a smooth affine group scheme $\underline{G}$ over $\mathbb{Z}_2$ with generic fiber $\mathrm{Aut}_{\mathbb{Q}_2} (L,Q)$,
 which satisfies $\underline{G}(\mathbb{Z}_2)=\mathrm{Aut}_{\mathbb{Z}_2} (L,Q)$.
Our method works for any unramified finite extension of $\mathbb{Q}_2$.
Therefore, we give a long awaited proof for the local density formula of  Conway  and  Sloane
and discover its generalization to unramified finite extensions of $\mathbb{Q}_2$.
As an example, we give the mass formula for the integral quadratic form $Q_n(x_1, \cdots, x_n)=x_1^2 + \cdots + x_n^2$ associated to a number field $k$ which is totally real and such that the ideal $(2)$ is unramified over $k$.
\end{abstract}

\maketitle


\section{Introduction}

\subsection{Introduction}
The problem of local densities has intrigued many great mathematicians, including Gauss and Eisenstein,
 Smith and Minkowski,  and Siegel.
 If $(L, Q)$ is a quadratic $R$-lattice, where $R$ is the ring of integers of a number field,
then the celebrated Smith-Minkowski-Siegel mass formula expresses the  mass of a quadratic lattice $(L, Q)$ as a product of local factors,
called the local densities of $(L, Q)$.
The local density is defined as the limit of a certain sequence.
G. Pall \cite{P} (for $p\neq 2$) and G. L. Watson \cite{Wa} (for $p=2$) computed this limit for an arbitrary lattice over $\mathbb{Z}_p$, thereby deriving an explicit formula for its local density.
For an expository sketch of their approach, see \cite{K}.
There is another proof of Y. Hironaka and F. Sato \cite{SH} computing the local density when $p\neq 2$.
They treat an arbitrary pair of lattices, not just a single lattice, over $\mathbb{Z}_p$ (for $p\neq 2$).
 J. H. Conway  and J. A. Sloane \cite{CS} further developed the formula for any $p$ and gave a heuristic explanation for it.
They mentioned in \cite{CS} regarding the local density formula for $p=2$  that ``Watson's formula seems to us to be essentially correct"
and  in the footnote that
``In fact we could not quite reconcile Watson's version with ours; they appear to differ by a factor of $2^n$ for $n$-dimensional forms. This is almost certainly due to our misunderstanding of Watson's conventions, which differ considerably from ours."
 In addition, they mentioned in the same paper that ``The reader may be confident that our version of the general mass formula is correct. $\cdots$ has enabled us to test the formula very stringently."
 Their formula   was computationally tested stringently and
 the proof of  Conway-Sloane's  local density formula, when $p=2$, has not been published in literature.

On the other hand, there is a simpler formulation of the local density as the integral of a certain volume form $\omega^{\mathrm{ld}}$
 over some open compact subgroup of an orthogonal group,
due to Kneser, Tamagawa, and Weil.
As explained in  the introduction of \cite{GY},  known methods unfortunately do not explain this formulation and involve complicated recursions.

Meanwhile, the work \cite{GY} by W. T. Gan and J.-K. Yu   is based on the existence of a smooth affine group scheme $\underline{G}$ over $\mathbb{Z}_p$ with generic fiber $\mathrm{Aut}_{\mathbb{Q}_p}(L, Q)$, which satisfies $\underline{G}(\mathbb{Z}_p)=\mathrm{Aut}_{\mathbb{Z}_p}(L, Q)$.
By constructing $\underline{G}$ explicitly and determining its special fiber,
they computed the integral and therefore obtained the formula for the local density when $p\neq 2$.

The main contribution of this paper is to construct $\underline{G}$ and investigate its special fiber
in order to get an explicit formula for the local density when $L$ is a quadratic $A$-lattice, where $A$ is an unramified finite extension of $\mathbb{Z}_2$.
Therefore, we give a long awaited proof for the local density formula of Conway and Sloane.
Furthermore, we discover its generalization to unramified finite extensions of $\mathbb{Q}_2$.
The special fiber of $\underline{G}$ has a large component group of the form $(\mathbb{Z}/2\mathbb{Z})^{N}$.
That is, we discover a large number of independent homomorphisms from the special fiber of $\underline{G}$ to the constant group $\mathbb{Z}/2\mathbb{Z}$.
Consequently, when replacing $\mathbb{Z}_2$ by an unramified finite extension with residue field $F_q$, where $q$ is a power of $2$, the $2$-power factor $2^M$ in the formula of Conway and Sloane  has to be replaced by $2^N \cdot q^{M-N}$. This fact is far from obvious from Conway-Sloane's explanation.

In  conclusion, this paper, combined with \cite{GY},
allows the computation of the mass formula for a quadratic $R$-lattice $(L, Q)$ when the ideal $(2)$ is unramified over $R$.

This paper is organized as follows.
 We first state the structural theorem for integral quadratic forms in Section 2. We then give an explicit construction of   $\underline{G}$
 (in Section 3) and its special fiber (in Section 4). Finally,
 by comparing $\omega^{\mathrm{ld}}$ and the canonical volume form $\omega^{\mathrm{can}}$ of $\underline{G}$,
 we obtain an explicit formula for the local density in Section 5.
  In Section 6, as an example, we give the mass formula for the integral quadratic form $Q_n(x_1, \cdots, x_n)=x_1^2 + \cdots + x_n^2$ associated to a number field $k$ which is totally real and such that the ideal $(2)$ is unramified over $k$.
 This formula is explicitly described using  the Dedekind zeta function of $k$ and a certain Hecke $L$-series.

 As in \cite{GY}, the smooth group schemes constructed in this paper should be of independent interest.

\subsection{Acknowledgements}
The author is deeply indebted to his Ph.D. thesis advisor Professor Jiu-Kang Yu for suggesting this problem and for giving many valuable ideas.
The author would like to thank Professor Benedict H. Gross for his interest in this work and his encouragement.
The author thanks the referees for their helpful comments, especially for pointing out a mistake and for suggesting a certain reference.
The author would like to thank Radhika Ganapathy, Hui Gao, Bogume Jang, Yean Su Kim, Elena Lawrick, Manish Mishra, Olivier  Taibi and Sandeep Varma
 for carefully reading a draft of this paper to help reduce the typographical errors and improve the presentation of this paper.

\section{Structural theorem for quadratic lattices and notations}

\subsection{Notations}

Notations and definitions in this subsection are taken from \cite{O1} and \cite{O2}.
Let $F$ be an unramified finite extension of $\mathbb{Q}_2$ with $A$ its ring of integers and $\kappa$ its residue field.

 We consider an $A$-lattice $L$ with a quadratic form $q:L \rightarrow A$.
We denote by a pair $(L, q)$ a quadratic lattice.
Let $\langle -,-\rangle_q$ be the symmetric bilinear form on $L$ such that
$$\langle x,y\rangle_q=\frac{1}{2}(q(x+y)-q(x)-q(y)).$$
We assume that $\langle x,y\rangle_q \in A$ and $V=L\otimes_A F$ is nondegenerate with respect to $\langle -,-\rangle_q$.

For any $\epsilon \in A$, we denote by $(\epsilon)$ the  $A$-lattice of rank 1 equipped with the symmetric bilinear form having Gram matrix $(\epsilon)$.
We use the symbol $ \epsilon\cdot A(\alpha, \beta)$ to denote the $A$-lattice $A\cdot e_1+A\cdot e_2$
with the symmetric bilinear form having Gram matrix $\epsilon\cdot \begin{pmatrix} \alpha&1\\ 1&\beta \end{pmatrix}$.

A quadratic lattice $L$ is the \textit{orthogonal sum} of sublattices $L_1$ and $L_2$, written $L=L_1\oplus L_2$, if $L_1\cap L_2=0$, $L_1$ is orthogonal to $L_2$ with respect to the symmetric bilinear form $\langle-,- \rangle_q$, and $L_1$ and $L_2$ together span $L$.

The fractional ideal generated by $q(X)$ as $X$ runs through $L$ will be called the \textit{norm} of $L$ and written $N(L)$.

By the \textit{scale} $S(L)$ of $L$, we mean the fractional ideal generated by the subset $\langle L,L \rangle_q$ of $F$.

The \textit{discriminant} of $L$, denoted by $d(L)$, is defined as the determinant of a Gram matrix defining the symmetric bilinear form $\langle -,- \rangle_q$,
up to multiplication by the square of a unit.

 We say $(L, q)$ is \textit{unimodular} if the discriminant $d(L)$ is a unit
 and a Gram matrix defining $\langle -,- \rangle_q$ has integral entries.

\begin{Def}
For a given quadratic lattice $L$,
\begin{enumerate}
\item[a)] A unimodular lattice $L$ is \textit{of parity type I} if $N(L)=A$; otherwise \textit{of parity type II}.
\item[b)] $(L, q)$ is \textit{modular} if $(L, a^{-1}q)$ is unimodular for some $a\in A \backslash \{0\}$,
where $a$ is unique up to a unit,
and in this case the parity type of $(L, q)$ is defined to be the parity type of $(L, a^{-1}q)$.
\item[c)] The zero lattice is considered to be \textit{modular of parity type II}.
\end{enumerate}
\end{Def}

\begin{Def}
We define the dual lattice of $L$, denoted by $L^{\perp}$, as
 $$L^{\perp}=\{x \in L\otimes_A F : \langle x, L\rangle_q \subset A \}.$$
\end{Def}

\begin{Rmk}
\begin{enumerate}
\item[a)] The scale $S(L)$ of a unimodular lattice $L$ is always $A$.
Based on the definition of a unimodular lattice,  the discriminant $d(L)$ is a unit and a Gram matrix defining $\langle -,- \rangle_q$ has integral entries.
Thus there is at least one unit entry of a Gram matrix  and this implies our claim.
\item[b)](93:15 in \cite{O2}) It is well known that a lattice of parity type I is diagonalizable.  In other words, $$L=\bigoplus_i (u_i),$$ where $u_i$'s are units.
\item[c)] If a unimodular lattice $L$ is of parity type II, then $$L=\bigoplus_i A(a_i, b_i),$$ where $a_i$ and $b_i$ are elements of the prime ideal $ (2)$ of $A$.
Thus the rank of $L$ is even.
\end{enumerate}
\end{Rmk}

\subsection{Structural theorem for quadratic lattices}

We  state the structural theorem for unimodular lattices as follows. Indeed, this theorem is a summary of some results from Section 93 of \cite{O2}.

\begin{Thm}
Assume that $L$ is \textit{unimodular}.
If $L$ is \textit{of parity type I}, then there is an orthogonal decomposition
$$L\cong \bigoplus_i A_i(0,0) \oplus K \oplus K^{\prime}.$$
Here, $A_i(0,0)=A(0,0)$, $K$ is empty or $A(2,\lambda)$ with $\lambda \in (2)$, and $K^{\prime}$ is $(\epsilon)$ or $A(1, 2\gamma)$
where $\epsilon\equiv 1$ mod $(2)$ and $\gamma\in A$.

     If $L$ is \textit{of parity type II}, then we have an orthogonal decomposition
     $$L\cong \bigoplus_i A_i(0,0) \oplus A(a,b).$$
     Here $A_i(0,0)=A(0,0)$ and  $a,b \in (2)$.
\end{Thm}

\begin{proof}
We use notations and a terminology from \cite{O2}.
It is easily seen that the set $q(L)+2(S(L))$ is an additive subgroup of $F$.
We let $M(L)$ denote the largest fractional ideal contained in the  group $q(L)+2(S(L))$.
Then we define the \textit{weight}  by
$2(M(L))+2(S(L)).$
We call the scalar $\textbf{a}$ a \textit{norm generator} of $L$ if $\textbf{a}\in q(L)+2(S(L))$ and $\textbf{a}A=N(L)$.
We call the scalar $\textbf{b}$ a \textit{weight generator} of $L$ if $\textbf{b}A=2(M(L))+2(S(L))$.
Let $d$ be the discriminant of $L$.
Regard $d$ as an element of the unit group $\textbf{u}$ of $A$.

We first state the following seven facts proved in  Example 93:10 and Section 93:18 of \cite{O2}.
\begin{itemize}
\item[i)] For every unit $u$ in the unit group $\textbf{u}$, there is a solution $v$ in $\textbf{u}$ such that $v^2\equiv u$ (mod $2$) because the residue field $\kappa$ is perfect and of characteristic 2.
Thus an $A$-lattice $(u)$ of rank 1 is isometric to an $A$-lattice $(\epsilon)$ of rank 1, where $\epsilon\equiv 1$ mod $(2)$.
\item[ii)] If dim $L \geqq 5$, then $$L \cong A(0,0)\oplus \cdots . $$
\item[iii)] We assume that dim $L = 4$ with $\mathrm{ord}_2\textbf{a}+\mathrm{ord}_2\textbf{b}$ odd.
Here, $\mathrm{ord}_2\textbf{a}$ (resp. $\mathrm{ord}_2\textbf{b}$) is the exponential order of $\textbf{a}$ (resp. $\textbf{b}$)
at the prime ideal $(2)$ in $A$.
We suppose that $d$ has been expressed in the form $d=1+\alpha$ with $\alpha \in \textbf{ab}A$.
For a given $\varrho \in A$ such that $4\varrho \in \textbf{ab}A$, consider the lattices
\[
  \begin{array}{lcr}
  J= A(\textbf{b},0) \oplus A(\textbf{a},-\alpha \textbf{a}^{-1}),\\
 J^{\prime}_{\varrho} = A(\textbf{b},4\varrho \textbf{b}^{-1}) \oplus A(\textbf{a}, -(\alpha-4\varrho)\textbf{a}^{-1}). 
  \end{array} .\]
Then $L$ is isomorphic to $J$ or  $J^{\prime}_{\varrho}$.
\item[iv)] If dim $L = 2$ with $\mathrm{ord}_2\textbf{a}+\mathrm{ord}_2\textbf{b}$ odd, we have the following:
$$L \cong A(\textbf{a}, \textbf{b} \varrho ) \textit{ for some $\varrho \in A$}.$$
\item[v)] Let dim $L = 3$ with $\mathrm{ord}_2\textbf{a}+\mathrm{ord}_2\textbf{b}$ odd.
Then we have
$$ L \cong A(\textbf{b},0)\oplus (-d)  $$
or
$$L\cong A(\textbf{b},4\varrho \textbf{b}^{-1})\oplus (-d(1-4\varrho)) \textit{ for some $\varrho \in A$ such that $4\varrho \in \textbf{b}A$}.$$
\item[vi)]  If dim $L = 2$ with $\mathrm{ord}_2\textbf{a}+\mathrm{ord}_2\textbf{b}$ even, we have
$$L \cong A(0,0 ) \mathrm{~or~} L \cong A(2,2\varrho) \textit{ for some $\varrho \in A$}. $$
\item[vii)] If dim $L \geqq 3$ and $\mathrm{ord}_2\textbf{a}+\mathrm{ord}_2\textbf{b}$ even, then
$$L\cong A(0,0)\oplus \cdots .\\$$
\end{itemize}

From ii), we may and do assume that the rank of $L$ is at most 4.

Assume that $L$ is of parity type I so that $N(L)=A$. Thus we may assume that $\textbf{a}=1$ and $\mathrm{ord}_2\textbf{a}=0$.
Since $2(S(L))\subseteqq 2(M(L))+2(S(L)) \subseteqq (2)$ and $2(S(L))=(2)$, we have the following equality:
$$(\textbf{b})=2(M(L))+2(S(L))=(2).$$
 Hence $\mathrm{ord}_2\textbf{b}=1$ and  $\mathrm{ord}_2\textbf{a}+\mathrm{ord}_2\textbf{b}$ is odd.
 We choose $\textbf{b}=2$.
If the rank of $L$ is odd (resp. even), the theorem follows from i) and v) (resp. from iii) and iv)).

 If $L$ is of parity type II so that $N(L)=(2)$,
  $\mathrm{ord}_2 \textbf{a}=\mathrm{ord}_2 \textbf{b}=1$ and $\mathrm{ord}_2\textbf{a}+\mathrm{ord}_2\textbf{b}$ is even.
 Choose $\textbf{a}=\textbf{b}=2$.
Then the theorem follows from vi) and vii).
\end{proof}

For a general lattice $L$, we have a Jordan splitting, namely $L=\bigoplus_i L_i$ such that $L_i$ is \textit{modular} and the sequence $\{s(i)\}_i$ increases,
where $(2^{s(i)})=S(L_i)$.
Unfortunately, a Jordan splitting of $L$ is not unique.
Nevertheless, we will attach certain well-defined quantities to $L$ at the end of this section.
We recall the following theorem   from \cite{O2}.

\begin{Thm}[(\cite{O2}, Section 91:9)]
Let
$$L=L_1\oplus \cdots \oplus L_t,~~~ L=K_1\oplus \cdots \oplus K_T$$
be two Jordan splittings of $L$ with $L_i, K_j$ non-zero for all $i, j$.
 Then $t=T$.
Furthermore, for $1 \leqq i \leqq t$, the scale, rank and  parity type of $L_i$ are the same as those of $K_i$.
\end{Thm}

If we allow $L_i$ to be the zero lattice, then we may assume
$S(L_i)=(2^i)$ without loss of generality.
We can rephrase this theorem as follows: Let $L=\bigoplus_i L_i$ be a Jordan splitting with $s(i)=i$ for all $i\geq 0$.
Then the scale, rank and parity type of $L_i$ depend only on $L$.
We will deal exclusively with a Jordan splitting satisfying $s(i)=i$ from now on.

\subsection{Lattices}

In this subsection, we will define several lattices and corresponding notations. Assume that a quadratic lattice $(L, q)$ is given and $S(L)=(2^l)$.
The following lattices will play a significant role in our construction of the smooth integral model:

\begin{itemize}
\item[(1)] $A_i=\{x\in L \mid \langle x,L\rangle_q \in 2^iA\}$.
\item[(2)] $X(L)$, the sublattice of $L$ such that
$X(L)/2L$ is the kernel of the symmetric bilinear form $\frac{1}{2^l}\langle-,-\rangle_q$ mod 2 on $L/2L$.
\item[(3)] $B(L)$, the sublattice of $L$ such that $B(L)/2L$ is the kernel of the linear form $\frac{1}{2^l}q$ mod 2 on $L/2L$.
\end{itemize}

To define  our integral structure a few more lattices will be needed, but we need some preparation to define them.

Assume $B(L)\varsubsetneq L$. Hence the bilinear form $\frac{1}{2^l}\langle-,-\rangle_q$ mod 2 on the $\kappa$-vector space $L/X(L)$
is nonsingular symmetric and non-alternating.
It is well known (Exercise 16 in Chapter 6 of \cite{KMRT}) that there is a unique vector $e \in L/X(L)$ such that
$(\frac{1}{2^l}\langle v,e \rangle_q)^2=\frac{1}{2^l}\langle v,v \rangle_q$ mod 2 for every vector $v \in L/X(L)$.
Let $\langle e\rangle$ denote  the 1-dimensional vector space spanned by the vector $e$ and denote by $ e^{\perp}$ the 1-codimensional subspace of $L/X(L)$ which is orthogonal to the vector $e$ with respect to  $\frac{1}{2^l}\langle -,-\rangle_q$ mod 2.
Then $$B(L)/X(L)=e^{\perp}.$$
If $B(L)= L$, then the bilinear form $\frac{1}{2^l}\langle-,-\rangle_q$ mod 2 on the $\kappa$-vector space $L/X(L)$
is nonsingular symmetric and alternating. In this case, we put $e=0\in L/X(L)$
and note that it is characterized by the same identity.\\

The remaining lattices we need for our definition are:
\begin{itemize}
\item[(4)] $W(L)$, the sublattice of $L$ such that $W(L)/X(L)=\langle e\rangle$.
\item[(5)] $Y(L)$, the sublattice of $L$ such that $Y(L)/2L$ is the kernel of the symmetric bilinear form $\frac{1}{2^l}\langle-,-\rangle_q$ mod 2 on $B(L)/2L$.
\item[(6)] $Z(L)$, the sublattice of $L$ such that $Z(L)/2L$ is the kernel of the quadratic form $\frac{1}{2^{l+1}}q$ mod 2 on $B(L)/2L$.
\end{itemize}

\begin{Rmk}
\begin{itemize}
\item[a)] We can associate
the 5 lattices above ($X(L), B(L), W(L), Y(L), Z(L)$)
to $(A_i, \frac{1}{2^i}q)$.  Denote the resulting lattices by $X_i,B_i,W_i,Y_i,Z_i.$
\item[b)] As $\kappa$-vector spaces, the dimensions of $A_i/B_i, W_i/X_i, Y_i/Z_i$ are at most 1.
\end{itemize}
\end{Rmk}

Let $L=\bigoplus_i L_i$ be a Jordan splitting. We assign a type to each $L_i$ as follows:
\[\left \{
  \begin{array}{l l}
   I & \quad  \textit{if $L_i$ is of parity type I};\\
   I^o & \quad \textit{if $L_i$ is of parity type I and the rank of $L_i$ is odd};\\
  I^e & \quad \textit{if $L_i$ is of parity type I and the rank of $L_i$ is even};\\
II & \quad  \textit{if $L_i$ is of parity type II}.\\
    \end{array} \right.\]
    In addition, we say that $L_i$ is
\[\left \{
  \begin{array}{l l}
  \textit{bound} & \quad  \textit{if at least one of  $L_{i-1}$ or $L_{i+1}$ is of parity type I};\\
 \textit{free} & \quad \textit{if both  $L_{i-1}$ and $L_{i+1}$ are of parity type II}.\\
    \end{array} \right.\]
Assume that a lattice $L_i$ is \textit{free} \textit{of type} $\textit{I}^e$.
We denote by $\bar{V_i}$ the $\kappa$-vector space $B_i/Z_i$.
Then we say that $L_i$ is
\[
 \left\{
  \begin{array}{l l}
  \textit{of type I}^e_1   & \quad  \text{if the dimension of  $\bar{V_i}$ is odd, equivalently, $X_i= Z_i$};\\
  \textit{of type I}^e_2   &   \quad  \text{otherwise}.\\
    \end{array} \right.
\]
We stress that we do not assign a type $\textit{I}^e_1$ or $\textit{I}^e_2$ to \textit{bound} lattices.
Notice that each type of $L_i$ is independent of the choice of a Jordan splitting.

\begin{Exa}
When $L$ is \textit{unimodular}, we assign a type to $L$ according to the shape of $K^{\prime}$ described in Theorem 2.4 as follows:
\[
      \begin{array}{c|c}
      K^{\prime} & \mathrm{type~of~} L \\
      \hline
      \emptyset & \textit{II}\\
      (\epsilon) & \textit{I}^o\\
       A(1, 2\gamma) \textit{ with $\gamma$ unit}   & \textit{I}^e_1\\
      A(1, 2\gamma) \textit{ with $\gamma\in (2)$} & \textit{I}^e_2\\
            \end{array}
    \]
\end{Exa}

\begin{Rmk}
We describe all these 6 lattices explicitly.
We use the following conventions. $\mathcal{L}_i$ denotes $\oplus_{j\neq i}2^{max\{0,i-j\}}L_j$.
We denote by $\mathcal{M}_i$ the $\bigoplus_i A_i(0,0) \oplus K$ of Theorem 2.4, for any $i$, when $L_i$ is \textit{of type I}.
When $L_i$ is \textit{of type II}, let $\mathcal{M}_i=L_i$.
Theorem 2.4 involves a basis for a lattice $K^{\prime}$, which we write as $\{e_1^{(i)}\}$ or $\{e_1^{(i)}, e_2^{(i)}\}$
if $L_i$ is \textit{of type $I^o$} or \textit{of type $I^e$}, respectively.

For all cases, we have $A_i=\mathcal{L}_i\oplus L_i$ and $X_i=\mathcal{L}_i\oplus 2L_i$.
If $L_i$ is \textit{of type I}, then the vector $e\in A_i/X_i$ is $(0, \cdots, 0, 1)$. Based on this, the lattices $B_i, W_i, Y_i$ are as follows:
\begin{center}
    \begin{tabular}{| l | l | l | l |}
    \hline
    Type & $B_i$ & $W_i$ & $Y_i$ \\ \hline
    $I^o$ & $\mathcal{L}_i\oplus \mathcal{M}_i \oplus 2A e_1^{(i)}$ & $\mathcal{L}_i\oplus \cdot 2 \mathcal{M}_i \oplus Ae_1^{(i)}$ & $X_i$ \\ \hline
  $I^e$ & $\mathcal{L}_i\oplus \mathcal{M}_i \oplus 2A e_1^{(i)}\oplus A e_2^{(i)}$ &
  $\mathcal{L}_i\oplus 2 \mathcal{M}_i \oplus 2Ae_1^{(i)} \oplus A e_2^{(i)}$& $W_i$ \\ \hline
$II$ & $A_i$ & $X_i$ & $X_i$\\ \hline
    \end{tabular}
\end{center}

The description of $Z_i$ is a little complicated.
Note that when $L_i$ is    \textit{free of type $I^e_1$} or \textit{bound}, the dimension of $Y_i/Z_i$ as a $\kappa$-vector space is \textit{1}.
We describe it case by case below.
Main idea is to start from $Y_i$ then to find a co-rank \textit{1} sublattice as the kernel of the associated linear form $\frac{1}{2^{i+1}}q$ mod $2$ on $Y_i/2Y_i$.

\begin{enumerate}
\item If $L_i$ is \textit{free of type $I^o$, $I^e_2$, or II}, then $Z_i=Y_i$.
\item If $L_i$ is \textit{free of type $I^e_1$}, then $Z_i=X_i$.
\item Let $\mathcal{J}=\{j\in \{i-1,  i+1\}| \textit{$L_j$ is of type I}\}$.
Let $L_i$ be \textit{bound of type $I^e$} and $\gamma_i$ be an element in $A$ such that $L_i=2^i\cdot\left(\mathcal{M}_i\oplus A(1, 2\gamma_i)\right)$.

If $\gamma_i$ is a unit in $A$, then
\[Z_i=\bigoplus_{j\notin \{i, i\pm 1\}}2^{max\{0, i-j\}}L_j\oplus 2^{max\{0, i-j\}}\left(\bigoplus_{j\in \{i\pm 1\}}\mathcal{M}_j\oplus Ae_2^{(j)}\right)
\oplus  \left(2\mathcal{M}_i\oplus 2Ae_1^{(i)}\right)\]
\[\oplus \left\{\left(\sum_{j\in \mathcal{J}}2^{max\{0, i-j\}}\cdot a_je_1^{(j)}\right)+\left(\sqrt{\gamma_i}\cdot a_ie_2^{(i)}\right)
|\textit{each }a_j, a_i\in A, \left(\sum_{j\in \mathcal{J}}a_j\right)+\sqrt{\gamma_i}\cdot a_i\in (2)\right\},\]
where the $e_2^{(j)}$  factor should be ignored for those $j \in\{i\pm 1\}$  such that $L_j$  is not of type $I^e$.
And $\sqrt{\gamma_i}$ is an element of $A$ such that $\sqrt{\gamma_i}$ mod $2$ $= \sqrt{\widetilde{\gamma_i}} (\in \kappa)$
where $\sqrt{\widetilde{\gamma_i}}$ is as explained at the beginning of Appendix.
If $\gamma_i$ is not a unit in $A$, then
\[Z_i=\bigoplus_{j\notin \{i, i\pm 1\}}2^{max\{0, i-j\}}L_j\oplus 2^{max\{0, i-j\}}\left(\bigoplus_{j\in \{i\pm 1\}}\mathcal{M}_j\oplus Ae_2^{(j)}\right)
\oplus  \left(2\mathcal{M}_i\oplus 2Ae_1^{(i)}\oplus Ae_2^{(i)}\right)\]
\[\oplus \left\{\sum_{j\in \mathcal{J}}2^{max\{0, i-j\}}\cdot a_je_1^{(j)}
|\textit{each }a_j\in A, \sum_{j\in \mathcal{J}}a_j\in (2)\right\},\]
where the $e_2^{(j)}$  factor should be ignored for those $j \in\{i\pm 1\}$  such that $L_j$  is not of type $I^e$.
\item
If $L_i$ is \textit{bound of type $I^o$ or $II$}, then
\[Z_i=\bigoplus_{j\notin \{i, i\pm 1\}}2^{max\{0, i-j\}}L_j\oplus 2^{max\{0, i-j\}}\left(\bigoplus_{j\in \{i\pm 1\}}\mathcal{M}_j\oplus Ae_2^{(j)}\right)
\oplus  \left(2\mathcal{M}_i\oplus 2Ae_1^{(i)}\right)\]
\[\oplus \left\{\sum_{j\in \mathcal{J}}2^{max\{0, i-j\}}\cdot a_je_1^{(j)}
|\textit{each }a_j\in A, \sum_{j\in \mathcal{J}}a_j\in (2)\right\},\]
where the $e_2^{(j)}$  factor should be ignored for those $j \in\{i\pm 1\}$  such that $L_j$  is not of type $I^e$,
and $e_1^{(i)}$ should be ignored if $L_i$ is  \textit{of type II}.
\end{enumerate}

\end{Rmk}

\begin{Rmk}
These 6 lattices have the following containment:
\begin{enumerate}
\item When $L_i$ is \textit{free of type $I^o$}, $A_i\varsupsetneq B_i \varsupsetneq Y_i=X_i=Z_i$, $A_i\varsupsetneq W_i \varsupsetneq Y_i$.
\item When $L_i$ is \textit{free of type $I^e_2$}, $A_i\varsupsetneq B_i \varsupsetneq W_i=Y_i=Z_i \varsupsetneq X_i$.
\item When $L_i$ is \textit{free of type $I^e_1$}, $A_i\varsupsetneq B_i \varsupsetneq W_i=Y_i \varsupsetneq X_i=Z_i$.
\item When $L_i$ is \textit{free of type $II$}, $A_i= B_i \varsupsetneq W_i=Y_i=X_i=Z_i$.
\item When $L_i$ is \textit{bound of type $I^o$}, $A_i\varsupsetneq B_i \varsupsetneq Y_i=X_i\varsupsetneq Z_i$, $A_i\varsupsetneq W_i \varsupsetneq Y_i$.
\item When $L_i$ is \textit{bound of type $I^e$}, $A_i\varsupsetneq B_i \varsupsetneq W_i=Y_i \varsupsetneq X_i$, $Y_i\varsupsetneq Z_i$.
\item When $L_i$ is \textit{bound of type $II$}, $A_i= B_i \varsupsetneq W_i=Y_i=X_i\varsupsetneq Z_i$.

\end{enumerate}
\end{Rmk}
From now on, the pair ($L,q $) is fixed throughout this paper and $\langle -,-\rangle$ denotes $\langle -,-\rangle_q$.

\section{The smooth integral model $\underline{G}$}

Let $\underline{G}^{\prime}$ be a naive integral model of the orthogonal group $\mathrm{O}(V, q)$, where $V=L\otimes_AF$, such that
for any commutative $A$-algebra $R$,
$$\underline{G}^{\prime}(R)=\mathrm{Aut}_{R}(L\otimes_AR, q\otimes_AR).$$
Let $\underline{G}$ be the smooth group scheme model of $\mathrm{O}(V, q)$
such that
$$\underline{G}(R)=\underline{G}^{\prime}(R)$$
for any \'etale $A$-algebra $R$.
Notice that $\underline{G}$ is uniquely determined with these properties by Proposition 3.7 in \cite{GY}.
For a detailed exposition of the relation between the local density of $(L, q)$ and $\underline{G}$, see Section 3 of \cite{GY}.

In this section, we give an explicit construction of the smooth integral model $\underline{G}$.
The construction of $\underline{G}$ is based on that of Section 5 in \cite{GY}.
Let $K=\mathrm{Aut}_A(L,q) \subset \mathrm{GL}_F(V)$, and $\bar{K}=\mathrm{Aut}_{A^{sh}}(L\otimes_A A^{sh},q\otimes_A A^{sh})$,
 where $A^{sh}$ is the strict henselization of $A$.
 To ease the notation, we say $g\in \bar{K}$ stabilizes a lattice $M\subseteq V$ if $g(M\otimes_AA^{sh})=M\otimes_AA^{sh}$.

\subsection{Main construction}

    In this subsection, we observe properties of elements of $\bar{K}$ and their matrix interpretation.
    We choose a Jordan splitting $L=\bigoplus_iL_i$ and, for each $i$, fix a basis of $L_i$ according to Theorem 2.4.
    Let $g$ be an element of $\bar{K}$.

   \begin{enumerate}
    \item[(1)] First of all, as explained in  Section 5.1 of \cite{GY},
    $g$ stabilizes the dual lattice $L^{\perp}$ of $L$.
    It is equivalent to saying that $g$ stabilizes $A_i$'s for every integer $i$.
    We interpret this fact in terms of matrices.\\
     Let $n_i=\mathrm{rank}_{A}L_i$, and $n=\mathrm{rank}_{A}L=\sum n_i$.
 Assume that $n_i=0$ unless $0\leq i <  N$.
 We always divide a matrix $g$ of size $n \times n$ into $N^2$ blocks such that the block in position $(i, j)$ is of size $n_i\times n_j$.
 For simplicity, the row and column numbering starts at $0$ rather than $1$.
The fact that $g$ stabilizes $L^{\bot}$ means that the $(i,j)$-block
 has entries in $2^{max\{0,j-i\}}A^{sh}$.\\
 From now on, we write
\[g= \begin{pmatrix} 2^{max\{0,j-i\}}g_{i,j} \end{pmatrix}.\]

 \item[(2)]
   $g$ stabilizes $A_i, B_i, W_i, X_i$ and induces the identity on $A_i/B_i$ and $W_i/X_i$.
 We also interpret these facts in terms of matrices as described below:
\begin{itemize}
\item[a)]     If $L_i$ is \textit{of type} $\textit{I}^o$, the diagonal $(i,i)$-block $g_{i,i}$ is of the form
     \[\begin{pmatrix} s_i&2y_i\\ 2v_i&1+2z_i \end{pmatrix}\in \mathrm{GL}_{n_i}(A^{sh}),\]
     where $s_i$ is an $(n_i-1) \times (n_i-1)$-matrix, etc.
\item[b)]      If $L_i$ is \textit{of type} $\textit{I}^e$, the diagonal $(i,i)$-block $g_{i,i}$ is of the form
     \[\begin{pmatrix} s_i&r_i&2t_i\\ 2y_i&1+2x_i&2z_i\\ v_i&u_i&1+2w_i \end{pmatrix}\in \mathrm{GL}_{n_i}(A^{sh}),\]
     where $s_i$ is an $(n_i-2) \times (n_i-2)$-matrix, etc.\\
     \end{itemize}

\item[(3)]
     $g$  stabilizes $Z_i$ and induces the identity on $W_i/(X_i\cap Z_i)$.
     To prove the latter, we choose an element $w$ in $W_i$.  It suffices to show that $gw-w\in X_i\cap Z_i$.
     By (2), it suffices to show that $gw-w\in Z_i$.
     This follows from the computation:

     $\frac{1}{2}\cdot \frac{1}{2^i}q(gw-w)=\frac{1}{2}(2\cdot \frac{1}{2^i}q(w)-2\cdot \frac{1}{2^i}\langle gw,w\rangle)=
     \frac{1}{2^i}(q(w)-\langle w+x,w\rangle)=\frac{1}{2^i}\langle x,w\rangle=0$ mod 2, where $gw=w+x$ for some $x\in X_i$.

    In terms of matrices, we have the following:
\begin{enumerate}
\item[a)] If $L_i$ is \textit{of type} $\textit{II}$ or \textit{free}, there is no new constraint.
     Thus we assume that $L_i$ is \textit{bound} \textit{of type} $\textit{I}$.
\item[b)]      If $L_{i-1}$ is \textit{of type} $\textit{I}^o$ and $L_{i+1}$ is \textit{of type} $\textit{II}$, then
     the $(n_{i-1}, n_i)^{th}$- entry of $g_{i-1, i}$ lies in the prime ideal (2).
    \item[c)]  If $L_{i-1}$ is \textit{of type} $\textit{I}^e$ and $L_{i+1}$ is \textit{of type} $\textit{II}$, then the $(n_{i-1}-1, n_i)^{th}$- entry of $g_{i-1, i}$ lies in  the prime ideal (2).
      \item[d)]  If $L_{i-1}$ is \textit{of type} $\textit{II}$ and $L_{i+1}$ is \textit{of type} $\textit{I}^o$, then the $(n_{i+1}, n_i)^{th}$- entry of $g_{i+1, i}$ lies in  the prime ideal (2).
     \item[e)]  If $L_{i-1}$ is \textit{of type} $\textit{II}$ and $L_{i+1}$ is \textit{of type} $\textit{I}^e$, then the $(n_{i+1}-1, n_i)^{th}$- entry of $g_{i+1, i}$ lies in  the prime ideal (2).
    \item[f)]  If $L_{i-1}$ and $L_{i+1}$ are \textit{of type} $\textit{I}^o$,
     then the sum of the $(n_{i-1}, n_i)^{th}$- entry of $g_{i-1, i}$ and the $(n_{i+1}, n_i)^{th}$- entry of $g_{i+1, i}$ lies in  the prime ideal (2).
     \item[g)]  If $L_{i-1}$ is \textit{of type} $\textit{I}^o$ and $L_{i+1}$ is \textit{of type} $\textit{I}^e$,
    then the sum of the $(n_{i-1}, n_i)^{th}$- entry of $g_{i-1, i}$ and the $(n_{i+1}-1, n_i)^{th}$- entry of $g_{i+1, i}$ lies in  the prime ideal (2).
     \item[h)]  If $L_{i-1}$ is \textit{of type} $\textit{I}^e$ and $L_{i+1}$ is \textit{of type} $\textit{I}^o$,
     then the sum of the $(n_{i-1}-1, n_i)^{th}$- entry of $g_{i-1, i}$ and
      the $(n_{i+1}, n_i)^{th}$- entry of $g_{i+1, i}$ lies in  the prime ideal (2).
     \item[i)]  If $L_{i-1}$ and $L_{i+1}$ are \textit{of type} $\textit{I}^e$,
     then the sum of the $(n_{i-1}-1, n_i)^{th}$- entry of $g_{i-1, i}$ and
     the $(n_{i+1}-1, n_i)^{th}$- entry of $g_{i+1, i}$ lies in  the prime ideal (2).\\
\end{enumerate}

 \item[(4)]
  The fact that $g$ induces the identity on $W_i/(X_i\cap Z_i)$ for all $i$ is equivalent to
  the fact that $g$ induces the identity on $(X_i\cap Z_i)^{\perp}/W_i^{\perp}$ for all $i$.

  We give another description of this condition.
  Since the space V has a non-degenerate bilinear form $\langle-,-\rangle$, V can be identified with its own dual.
 We define the adjoint $g^{\ast}$ characterized by $\langle gv,w \rangle=\langle v, g^{\ast}w  \rangle$.
 Then the fact that $g$ induces the identity on $(X_i\cap Z_i)^{\perp}/W_i^{\perp}$ is
 the same as the fact that $g^{\ast}$ induces the identity on $W_i/(X_i\cap Z_i)$.

To interpret the above in terms of matrices, we notice that $g^{\ast}$ is an element of $\bar{K}$ as well.
 Thus if we apply (3) to $g^{\ast}$, we have the following:
\begin{enumerate}
\item[a)]  If $L_i$ is \textit{of type} $\textit{II}$ or \textit{free}, then there is  no new constraint.
Thus we assume that $L_i$ is \textit{bound} \textit{of type} $\textit{I}$.
    \item[b)]  If $L_{i-1}$ is \textit{of type} $\textit{I}$, $L_i$ is \textit{of type} $\textit{I}^o$ and $L_{i+1}$ is \textit{of type} $\textit{II}$,
     then the $(n_i, n_{i-1})^{th}$- entry of $g_{i, i-1}$ lies in  the prime ideal (2).
     \item[c)]  If $L_{i-1}$ is \textit{of type} $\textit{I}$, $L_i$ is \textit{of type} $\textit{I}^e$ and $L_{i+1}$ is \textit{of type} $\textit{II}$,
     then the $(n_i-1, n_{i-1})^{th}$- entry of $g_{i, i-1}$ lies in  the prime ideal (2).
     \item[d)] If $L_{i-1}$ is \textit{of type} $\textit{II}$, $L_i$ is \textit{of type} $\textit{I}^o$ and $L_{i+1}$ is \textit{of type} $\textit{I}$,
     then the $(n_i, n_{i+1})^{th}$- entry of $g_{i, i+1}$ lies in  the prime ideal (2).
     \item[e)]  If $L_{i-1}$ is \textit{of type} $\textit{II}$, $L_i$ is \textit{of type} $\textit{I}^e$ and $L_{i+1}$ is \textit{of type} $\textit{I}$,
     then the $(n_i-1, n_{i+1})^{th}$- entry of $g_{i, i+1}$ lies in  the prime ideal (2).
     \item[f)] If $L_{i-1}$ and $L_{i+1}$ are \textit{of type} $\textit{I}$ and $L_i$ is \textit{of type} $\textit{I}^o$,
     then the sum of the $(n_i, n_{i-1})^{th}$- entry of $g_{i, i-1}$ and the $(n_i, n_{i+1})^{th}$- entry of $g_{i, i+1}$ lies in  the prime ideal (2).
     \item[g)]  If $L_{i-1}$ and $L_{i+1}$ are \textit{of type} $\textit{I}$ and $L_i$ is \textit{of type} $\textit{I}^e$, then
     the sum of the $(n_i-1, n_{i-1})^{th}$- entry of $g_{i, i-1}$ and the $(n_i-1, n_{i+1})^{th}$- entry of $g_{i, i+1}$ lies in  the prime ideal (2).\\
\end{enumerate}

   \item[(5)] By combining (3) and (4), we obtain the following:

\begin{enumerate}
      \item[a)]  If $L_i$ and $L_{i+1}$ are \textit{of type} $\textit{I}^o$, the $(n_i, n_{i+1})^{th}$ (resp. $(n_{i+1}, n_i)^{th}$)- entry of $g_{i, i+1}$
      (resp. $g_{i+1, i}$) lies in the prime ideal (2).
    \item[b)]  If $L_i$ and $L_{i+1}$ are \textit{of type} $\textit{I}^e$,
    the $(n_i-1, n_{i+1})^{th}$ (resp. $(n_{i+1}-1, n_i)^{th}$)- entry of $g_{i, i+1}$ (resp. $g_{i+1, i}$) lies in the prime ideal (2).
    \item[c)]  If $L_i$ is \textit{of type} $\textit{I}^o$ and $L_{i+1}$ is \textit{of type} $\textit{I}^e$,
    the $(n_i, n_{i+1})^{th}$ (resp. $(n_{i+1}-1, n_i)^{th}$)- entry of $g_{i, i+1}$ (resp. $g_{i+1, i}$) lies in  the prime ideal (2)
   \item[d)]  If $L_i$ is \textit{of type} $\textit{I}^e$ and $L_{i+1}$ is \textit{of type} $\textit{I}^o$,
    the $(n_i-1, n_{i+1})^{th}$ (resp. $(n_{i+1}, n_i)^{th}$)- entry of $g_{i, i+1}$ (resp. $g_{i+1, i}$) lies in  the prime ideal (2).\\
\end{enumerate}

\item[(6)] Consequently, we have the following matrix form for $g$:
  $$g= \begin{pmatrix} 2^{max\{0,j-i\}}g_{i,j} \end{pmatrix},$$
 where $g_{i,i}$ is as described in  (2), and $g_{i,i+1}$ and $g_{i+1,i}$ are as described in  (5).
\end{enumerate}

\subsection{Construction of $\underline{M}^{\ast}$}

We first state the following lemma.
\begin{Lem}
Let $V$ be an $F$-vector space.
Let $\{\phi_i\}\subset \mathrm{Hom}_F(V, F)$ be a finite set $F$-spanning  $\mathrm{Hom}_F(V, F)$.
Define a functor from the category of commutative flat $A$-algebras to the category of sets as follows:
\[X : R \mapsto \{x\in V\otimes_AR \mid (\phi_i\otimes_AR)(x)\in R \mathrm{~for~all~}i\}.\]
Then $X$ is representable by the affine space $\textbf{A}^{\mathrm{dim~}_FV}$ over $A$ of dimension $\mathrm{dim~}_FV$.
\end{Lem}

\begin{proof}
Let $\{\psi_1, \cdots, \psi_m\}$
be a basis of an $A$-span of $\{\phi_i\}$.
Let $v_1, \cdots, v_m$ be the dual basis.
Then we have the following:
\[X(R)=\oplus_i Rv_i.\]
This completes the proof.
\end{proof}

    We define a functor from the category of commutative flat $A$-algebras to the category of monoids as follows.
    For any commutative flat $A$-algebra $R$, set
    $$m\in \underline{M}(R) \textit{ if and only if } m \in \mathrm{End}_{R}(L \otimes_A R)$$ with the following conditions:
\begin{enumerate}
   \item[(1)]   $m$ stabilizes $A_i\otimes_A R,B_i\otimes_A R,W_i\otimes_A R,X_i\otimes_A R,Y_i\otimes_A R,Z_i\otimes_A R$ for all $i$.
   \item[(2)]   $m $ induces the identity on $A_i\otimes_A R/ B_i\otimes_A R, W_i\otimes_A R/(X_i\cap Z_i)\otimes_A R,
    (X_i\cap Z_i)^{\perp}\otimes_A R/W_i^{\perp}\otimes_A R\mathrm{~for~all~}i$.
\end{enumerate}

    Then by the above lemma, the functor $\underline{M}$ is representable by a unique flat $A$-algebra $A(\underline{M})$ which is a polynomial ring over $A$ in $n^2$ variables.
    Moreover, it is easy to see that $\underline{M}$ has the structure of a scheme in monoids by showing that $\underline{M}(R)$ is  closed under multiplication.

    We stress that the above description of $\underline{M}(R)$, the set of $R$-points on the scheme $\underline{M}$, is no longer true when $R$ is a $\kappa$-algebra.
    Now suppose that $R$ is a $\kappa$-algebra.
    By choosing a basis for  $L$ as in Section 3.1, we describe each element of $\underline{M}(R)$ formally as a matrix
    $\begin{pmatrix} 2^{max\{0,j-i\}}m_{i,j} \end{pmatrix}$, where $m_{i,j}$ is an  $(n_i \times n_j)$-matrix with entries in $R$  as described in Section 3.1.
    To multiply $(m_{i,j})$ and $(m^{\prime}_{i,j})$, we refer to the description of Section 5.3 in \cite{GY}.

 Let $d$ be the determinant of the matrix $\begin{pmatrix} 2^{max\{0,j-i\}}m_{i,j} \end{pmatrix}$.
 Then $\mathrm{Spec} (A(\underline{M})_d)$ is an open subscheme of $\underline{M}$.

 We define a functor from the category of commutative $A$-algebras to the category of groups as follows.
For $m\in \underline{M}(R)$ with $R$ a commutative $A$-algebra,
 set
 $$\underline{M}^{\ast}(R)=\{ m \in \underline{M}(R) :  \textit{there exists $m^{-1}\in \underline{M}(R)$ such that $m\cdot m^{-1}=m^{-1}\cdot m=1$}\}.$$
 We claim that $\underline{M}^{\ast}$ is representable by $\mathrm{Spec} (A(\underline{M})_d)$.
 Notice that $\mathrm{Spec} (A(\underline{M})_d)(R)$, the set of $R$-points of $\mathrm{Spec} (A(\underline{M})_d)$ for a flat $A$-algebra $R$, is characterized by
 $$\{m\in\underline{M}(R):
   \textit{there exists $\widetilde{m}^{-1} \in \mathrm{End}_R(L\otimes_AR)$ such that $m\cdot \widetilde{m}^{-1}=\widetilde{m}^{-1}\cdot m=1$} \}.$$
   Note that $\underline{M}(R) \subset \mathrm{End}_R(L\otimes_AR)$ for a flat $A$-algebra  $R$.

 We first show that  $\widetilde{m}^{-1} ~ (\in \mathrm{End}_R(L\otimes_AR))$ with $m\in \underline{M}(R)$ is an element of $\underline{M}(R)$ for every flat $A$-algebra $R$.
 To verify this statement, it suffices to show that $\widetilde{m}^{-1}$ satisfies conditions (1) and (2) defining $\underline{M}$
 and it follows from Lemma 3.2 which will be stated below.
For any flat $A$-algebra $R$, we consider the following well-defined map:
$$\mathrm{Spec} (A(\underline{M})_d)(R) \longrightarrow \mathrm{Spec} (A(\underline{M})_d)(R), m \mapsto \widetilde{m}^{-1}.$$
Since $\mathrm{Spec} (A(\underline{M})_d)$ is affine and flat, this map is represented by a morphism as schemes.
This implies that if $m\in \mathrm{Spec} (A(\underline{M})_d)(R)$ then $\widetilde{m}^{-1}\in \underline{M}(R)$ for any commutative $A$-algebra $R$.
Therefore, $$\mathrm{Spec} (A(\underline{M})_d)(R)=\underline{M}^{\ast}(R)$$ for any commutative $A$-algebra $R$.
We now state Lemma 3.2.
 \begin{Lem}
Let $L^{\prime}$ be a sublattice of $L$ and $m$ be an element of $\mathrm{Spec} (A(\underline{M})_d)(R)$, where $R$ is a flat $A$-algebra.
Assume that $m$ stabilizes $L^{\prime}\otimes_AR$. Then this lattice $L^{\prime}\otimes_AR$ is stabilized by $\widetilde{m}^{-1}$ as well.
 \end{Lem}

\begin{proof}
Since $m$ stabilizes $L^{\prime}\otimes_AR$, we have that $m\cdot L^{\prime}\otimes_AR \subset L^{\prime}\otimes_AR$.
In addition, $m$ is an element of $\mathrm{Spec} (A(\underline{M})_d)(R)$ and so the determinant of $m$ is a unit in $R$.
Therefore, $m\cdot L^{\prime}\otimes_AR = L^{\prime}\otimes_AR$.
This implies that $\widetilde{m}^{-1}\cdot L^{\prime}\otimes_AR = L^{\prime}\otimes_AR$.
\end{proof}

Therefore, we conclude that $\underline{M}^{\ast}$ is an open subscheme of $\underline{M}$, with generic fiber $M^{\ast}=GL_{F}(V)$,
and that $\underline{M}^{\ast}$ is smooth over $A$. Moreover, $\underline{M}^{\ast}$ is a group scheme since $\underline{M}$ is a scheme in monoids.

\begin{Rmk}
We give another description for the functor $\underline{M}$.
Let us define a functor from the category of commutative flat $A$-algebras to the category of rings as follows:

For any commutative flat $A$-algebra $R$, set
$$\underline{M}^{\prime}(R) \subset \{m \in \mathrm{End}_{R}(L \otimes_A R) \}$$ with the following conditions:
\begin{itemize}
    \item[(1)] $m$ stabilizes $A_i\otimes_A R,B_i\otimes_A R,W_i\otimes_A R,X_i\otimes_A R,Y_i\otimes_A R,Z_i\otimes_A R$ for all $i$.
     \item[(2)] $m $ maps $A_i\otimes_A R, W_i\otimes_A R, (X_i\cap Z_i)^{\perp}\otimes_A R$ into $B_i\otimes_A R, (X_i \cap Z_i)\otimes_A R, W_i^{\perp}\otimes_A R$, respectively.
\end{itemize}

    Then the functor $\underline{M}$ is the same as the functor $1+\underline{M}^{\prime}$,
    where $(1+\underline{M}^{\prime})(R)=\{1+m : m \in \underline{M}^{\prime}(R) \}$.
\end{Rmk}

 \subsection{Construction of $\underline{Q}$}

 Recall that the pair ($L, q$) is fixed throughout this paper and the lattices $A_i$, $B_i$, $W_i$, $X_i$, $Y_i$, $Z_i$ only depend on the quadratic pair $(L, q)$.
  For any flat $A$-algebra $R$, let $\underline{Q}(R)$ be the set of quadratic forms $f$ on $L\otimes_{A}R$
  such that $S(L,f)\subseteq S(L,q)$ and $f$ satisfies the following conditions:
\begin{enumerate}
 \item[a)] $\langle L\otimes_{A}R,A_i\otimes_{A}R\rangle_f \subset 2^iR$ for all $i$.
  \item[b)] $X_i\otimes_{A}R/2A_i\otimes_{A}R$ is contained in the kernel of the symmetric bilinear form $\langle-,-\rangle_{f,i}$ mod 2
 on $A_i\otimes_{A}R/2A_i\otimes_{A}R$. Here, $\langle -,- \rangle_{f,i}=\frac{1}{2^i}\langle -,- \rangle_f$.
  \item[c)] $B_i\otimes_{A}R/2A_i\otimes_{A}R$ is contained in the kernel of the linear form $\frac{1}{2^i}f$ mod 2 on $A_i\otimes_{A}R/2A_i\otimes_{A}R$.
  \item[d)] Assume $B_i\varsubsetneq A_i$. We have seen the existence of the unique vector $e \in A_i/X_i$
 such that $\langle v,e \rangle_{q,i}^2=\langle v,v \rangle_{q,i}$ mod 2 for every vector $v \in A_i/X_i$ in Section 2.3.
 Then $e\otimes 1 \in A_i \otimes_{A}R/X_i\otimes_{A}R$ also satisfies the condition that
 $\langle v,e\otimes 1 \rangle_{f,i}^2=\langle v,v \rangle_{f,i}$ mod 2 for every vector $v \in A_i\otimes_{A}R/X_i\otimes_{A}R$.
  \item[e)] $Y_i\otimes_{A}R/2A_i\otimes_{A}R$ is contained in the kernel of the symmetric bilinear form
 $\langle -,- \rangle_{f,i}$ mod 2 on $B_i\otimes_{A}R/2A_i\otimes_{A}R$.
  \item[f)] $Z_i\otimes_{A}R/2A_i\otimes_{A}R$ is contained in the kernel of the quadratic form $\frac{1}{2}\cdot \frac{1}{2^i} f$ mod 2 on $B_i\otimes_{A}R/2A_i\otimes_{A}R$.
  \item[g)] $\frac{1}{2^i} f$ mod 2 = $\frac{1}{2^i} q$ mod 2 on $A_i \otimes_{A}R/2A_i\otimes_{A}R$.
  \item[h)] $\frac{1}{2^i} f(w_i)-\frac{1}{2^i} q(w_i) \in (4)$, where $w_i \in W_i\otimes_{A}R$.
  \item[i)] $\langle a_i, w_i\rangle_{f,i} \equiv \langle a_i, w_i\rangle_{q,i}$ mod 2, where $a_i \in A_i\otimes_{A}R$ and $w_i \in W_i\otimes_{A}R$
  \item[j)] $\langle w_i^{\prime}, w_i\rangle_{f} \in R$ and $\langle z_i^{\prime}, z_i\rangle_{f} \in R$.
 Here, $w_i^{\prime} \in W_i^{\perp}\otimes_{A}R, w_i \in W_i\otimes_{A}R$ and $z_i^{\prime} \in (X_i\cap Z_i)^{\perp}\otimes_{A}R, z_i \in (X_i\cap Z_i)\otimes_{A}R$.
 In addition,
 $\langle w_i, z_i^{\prime}\rangle_{f}-\langle w_i, z_i^{\prime}\rangle_{q} \in R$.
 \end{enumerate}

We interpret the above conditions in terms of matrices.
For a flat $A$-algebra $R$, $\underline{Q}(R)$ is
the set of symmetric matrices $$\begin{pmatrix}2^{max\{i,j\}}f_{i,j}\end{pmatrix}$$ of size $n\times n$ satisfying the following:
 \begin{enumerate}
\item[(1)] The size of $f_{i,j}$ is $n_i\times n_j$.
\item[(2)] If $L_i$ is \textit{of type} $\textit{I}^o$ with respect to $q$, then $f_{i,i}$ is of the form $$\begin{pmatrix} a_i&2b_i\\ 2\cdot{}^tb_i&\epsilon +4c_i \end{pmatrix}.$$
 Here, the diagonal entries of $a_i$ are $\equiv 0  \ \ \mathrm{mod}\ \ 2$, where $a_i$ is an $(n_i-1) \times (n_i-1)$-matrix, etc.
\item[(3)] If $L_i$ is \textit{of type} $\textit{I}^e$, then $f_{i,i}$ is of the form
   $$\begin{pmatrix} d_i&{}^tb_i&2e_i\\ b_i&1+2a_i&1+2c_i \\ 2\cdot{}^te_i&1+2c_i&2\gamma_i+4f_i \end{pmatrix}.$$
Here,  the diagonal entries of $d_i$ are $\equiv 0  \ \ \mathrm{mod}\ \ 2$, where $d_i$ is an $(n_i-2) \times (n_i-2)$-matrix, etc.
\item[(4)]  If $L_i$ is \textit{of type} $\textit{II}$, then  the diagonal entries of $f_{i,i}$ are $\equiv 0  \ \ \mathrm{mod}\ \ 2$.
\item[(5)]   If $L_i$ and $L_{i+1}$ are \textit{of type} $\textit{I}$, then the $(n_i, n_{i+1})^{th}$-entry of $f_{i, i+1}$ lies in the ideal (2).
 \end{enumerate}

  It is easy to see, by Lemma 3.1, that $\underline{Q}$ is represented by a flat $A$-scheme which is isomorphic to an affine space of dimension $(n^2+n)/2$.
      Note that our fixed quadratic form $q$ is an element of $\underline{Q}(A)$.

\subsection{Smooth affine group scheme $\underline{G}$}

\begin{Thm}
    For any flat $A$-algebra $R$, the group $\underline{M}^{\ast}(R)$ acts on the right of $\underline{Q}(R)$
    by $f\circ m = {}^tm\cdot f\cdot m$. Then this action is represented by an action morphism of schemes
     \[\underline{Q} \times \underline{M}^{\ast} \longrightarrow \underline{Q} .\]
\end{Thm}

\begin{proof}
 We start with any $m\in \underline{M}^{\ast}(R)$ and $f\in \underline{Q}(R)$.
It suffices to show that $f\circ m$ satisfies  conditions a) to j) given in Section 3.3.\\
From the construction of $\underline{M}^{\ast}$, it is obvious that $f\circ m$ satisfies conditions a) to f).\\
Condition g) is obvious from the fact that $m$  stabilizes $A_i$ and $B_i$ and induces the identity on $A_i/B_i$.\\
The fact that $m$ induces the identity on $W_i/X_i$ and $W_i/Z_i$ implies that $f\circ m$ satisfies  condition h).\\
For  condition i), it suffices to show that $\langle ma_i, mw_i  \rangle_{f,i} \equiv \langle a_i, w_i  \rangle_{q,i}$ mod 2.
We denote $ma_i=a_i+b_i$ and $mw_i=w_i+x_i$, where $b_i\in B_i\otimes_A R, x_i \in X_i\otimes_A R$.
Now it suffices to show $\langle a_i+b_i, x_i  \rangle_{f,i} +\langle b_i, w_i  \rangle_{f,i} \equiv 0 \mathrm{~mod~}2$.
 Firstly, $\langle a_i+b_i, x_i  \rangle_{f,i} \equiv 0 \mathrm{~mod~}2$ by the definition of the lattice $X_i$.
Secondly, we have $\langle b_i, w_i  \rangle_{f,i}\mathrm{~mod~}2 \equiv 0$ by observing that
$\langle b_i, e  \rangle_{f,i}^2  \equiv \langle b_i, b_i  \rangle_{f,i} \equiv 0 \mathrm{~mod~}2$ when $B_i\varsubsetneq A_i$,
where $e$ is the unique vector chosen earlier. If $B_i=A_i$,  it is obvious because $W_i=X_i$.\\
For  condition j), it suffices to prove that $\langle mw_i, mz_i^{\prime}\rangle_{f}-\langle w_i, z_i^{\prime}\rangle_{q} \in R$.
Notice that $mw_i=w_i+z_i$ and $mz_i^{\prime}=z_i^{\prime}+w_i^{\prime}$, where $z_i\in (X_i\cap Z_i)\otimes_AR$ and $w_i^{\prime}\in W_i^{\perp}$.
Thus, this can be easily seen by the fact that
$\langle w_i, w_i^{\prime}\rangle_{f}+\langle z_i, z_i^{\prime}\rangle_{f}+\langle z_i, w_i^{\prime}\rangle_{f}\in R$.
\end{proof}

     \begin{Thm}
     Let $\rho$ be the morphism $\underline{M}^{\ast} \rightarrow \underline{Q}$ defined by $\rho(m)=q \circ m$.
     Then $\rho$ is smooth of relative dimension dim $\mathrm{O}(V,q)$.
     \end{Thm}

\begin{proof}
   The theorem follows from Theorem 5.5 of \cite{GY} and the following lemma.
\end{proof}

      \begin{Lem}
      The morphism $\rho \otimes \kappa : \underline{M}^{\ast}\otimes \kappa \rightarrow \underline{Q}\otimes \kappa$
      is smooth of relative dimension dim $\mathrm{O}(V,q)$.
      \end{Lem}

    \begin{proof}
    The proof is based on Lemma 5.5.2 in \cite{GY}.
    It is enough to check the statement over the algebraic closure $\bar{\kappa}$ of $\kappa$.
    By \cite{H}, III.10.4, it suffices to show that, for any $m \in \underline{M}^{\ast}(\bar{\kappa})$,
    the induced map on the Zariski tangent space $\rho_{\ast, m}:T_m \rightarrow T_{\rho(m)}$ is surjective.

   We define two functors from the category of commutative flat $A$-algebras to the category of abelian groups as follows:
     \[T_1(R)=\{m-1 : m\in\underline{M}(R)\},\]
     \[T_2(R)=\{f-q : f\in\underline{Q}(R)\}.\]

   The functor $T_1$ (resp. $T_2$) is representable by a flat $A$-algebra which is a polynomial ring over $A$ in $n^2$ (resp. $(n^2+n)/2$) variables.
    Moreover, they have the structure of a commutative group scheme since they are closed under addition.
    In fact, $T_1$ is the same as the functor $\underline{M}^{\prime}$ in Remark 3.3.

   We still need to introduce another functor on flat $A$-algebras.
   Define $T_3(R)$ to be the set of all $(n \times n)$-matrices $y$ over $R$ with the following conditions:
    \begin{enumerate}
  \item[a)] The $(i,j)$-block $y_{i,j}$ of $y$ has entries in $2^{max(i,j)}R$ so that
   $$y=\begin{pmatrix} 2^{max(i,j)}y_{i,j}\end{pmatrix}.$$
   Here, the size of $y_{i,j}$ is $n_i\times n_j$.
   \item[b)] If $L_i$ is \textit{of type} $\textit{I}^o$, $y_{i,i}$ is of the form
   \[\begin{pmatrix} s_i&2y_i\\ 2v_i&2z_i \end{pmatrix}\in \mathrm{M}_{n_i}(R),\]
     where $s_i$ is an $(n_i-1) \times (n_i-1)$-matrix, etc.
 \item[c)] If $L_i$ is \textit{of type} $\textit{I}^e$, $y_{i,i}$ is of the form
   \[\begin{pmatrix} s_i&r_i&2t_i\\ y_i&x_i&2z_i\\ 2v_i&2u_i&2w_i \end{pmatrix}\in \mathrm{M}_{n_i}(R),\]
     where $s_i$ is an $(n_i-2) \times (n_i-2)$-matrix, etc.
    \item[d)] If $L_i$ and $L_{i+1}$ are \textit{of type} $\textit{I}$,
     then both the $(n_i, n_{i+1})^{th}$- entry of $y_{i, i+1}$ and the $(n_{i+1}, n_i)^{th}$- entry of $y_{i+1, i}$ lie in the ideal (2).
      \end{enumerate}
It is easy to see that the functor $T_3$ is represented by a flat $A$-scheme.

   Then we identify $T_m$ with $T_1(\bar{\kappa})$ and $T_{\rho(m)}$ with $T_2(\bar{\kappa})$.
   The map $\rho_{\ast, m}:T_m \rightarrow T_{\rho(m)}$ is then $X \mapsto m^t\cdot q\cdot X + X^t\cdot q\cdot m$, where the sum and the multiplication are to be interpreted as in Section 5.3 of \cite{GY}.

   To prove surjectivity, it suffices to show the following three statements:
       \begin{itemize}
   \item[(1)] $X \mapsto q\cdot X $ is a bijection $T_1(\bar{\kappa}) \rightarrow T_3(\bar{\kappa})$;
   \item[(2)] for any $m \in \underline{M}^{\ast}(\bar{\kappa})$, $Y \mapsto {}^t m \cdot Y$ is a bijection from $T_3(\bar{\kappa})$ to itself;
   \item[(3)] $Y \mapsto {}^t Y + Y$ is a surjection $T_3(\bar{\kappa}) \rightarrow T_2(\bar{\kappa})$.
    \end{itemize}

   (3) is direct from the construction of $T_3(\bar{\kappa})$. Hence we provide the proof of (1) and (2).

For (1), we first observe that two functors $T_1$ and $T_3$ are representable by flat affine schemes.
Therefore, it suffices to show that the map
$$T_1(R)\longrightarrow T_3(R), X\mapsto q\cdot X$$
is bijective for a flat $A$-algebra $R$. To prove this, it suffices to show that
the map $T_1(R) \rightarrow T_3(R), X \mapsto q\cdot X $
   and the map $T_3(R) \rightarrow T_1(R), Y \mapsto q^{-1}\cdot Y $ are  well-defined for all flat $A$-algebra $R$.

    For the first map, it suffices to show that $q\cdot X$ satisfies the four conditions defining the functor $T_3$.
     We represent the given quadratic form $q$ by a symmetric matrix $\begin{pmatrix} 2^{i}\cdot \delta_i\end{pmatrix}$ with $2^{i}\cdot \delta_i$
     for the $(i,i)$-block and $0$ for remaining blocks.
  We express $$X=\begin{pmatrix} 2^{max\{0,j-i\}}x_{i,j} \end{pmatrix}.$$
  Then
   $$q\cdot X = \begin{pmatrix} 2^{max(i,j)}y_{i,j}\end{pmatrix}.$$
  Here, $y_{i,i}=\delta_i\cdot x_{i,i}$, $y_{i, i+1}= \delta_i\cdot x_{i,i+1}$ and $y_{i+1, i}= \delta_{i+1}\cdot x_{i+1, i}$.
  These matrix equations are easily computed and so we conclude $q\cdot X \in T_3(R)$.

   For the second map, we express $Y=\begin{pmatrix} 2^{max(i,j)}y_{i,j}\end{pmatrix}$ and $q^{-1}=\begin{pmatrix} 2^{-i}\cdot \delta_i^{-1}\end{pmatrix}$.
   Then we have the following:
   $$q^{-1}\cdot Y = \begin{pmatrix} 2^{max\{0,j-i\}}x_{i,j} \end{pmatrix}.$$
  Here, $x_{i,i}=\delta_i^{-1}\cdot y_{i,i}$,
   $ x_{i, i+1}=\delta_i^{-1}\cdot y_{i, i+1}$ and
  $x_{i+1, i}=\delta_{i+1}^{-1}\cdot y_{i+1, i}$.
   From these, it is easily checked that $q^{-1}\cdot Y$ is an element of $T_1(R)$.

   For (2), it suffices to show that the map
   $$T_3(\bar{\kappa}) \rightarrow T_3(\bar{\kappa}), Y \mapsto {}^t m \cdot Y,  \mathrm{~for~any~} m \in \underline{M}^{\ast}(\bar{\kappa}),$$
   is well-defined so that its inverse map $Y \mapsto {}^t m^{-1} \cdot Y$ is  well-defined as well.

   We again express $m=\begin{pmatrix} 2^{max\{0,j-i\}}m_{i,j} \end{pmatrix}$ and $Y=\begin{pmatrix} 2^{max(i,j)}y_{i,j}\end{pmatrix}$.
   We need to show that ${}^t m \cdot Y $ satisfies four conditions defining the functor $T_3$.

   Condition a) is  explained in Lemma 5.5.2 of \cite{GY}.
   For the second and the third,  we observe that the diagonal $(i,i)$-block of ${}^t m \cdot Y $ is
   $$2^i\cdot {}^t m_{i,i}\cdot y_{i,i}+\sum_{j\neq i}2^{max\{0,i-j\}+max\{j,i\}}\cdot {}^t m_{j,i}\cdot y_{i,j}.$$
   Notice that $max\{0,i-j\}+max\{j,i\}$ is greater than $i$ if $j\neq i$.
   By observing the above  equation, it is easily seen that ${}^t m \cdot Y $ satisfies  conditions b) and c).

   Finally we observe that the $(i,i+1)$-block of ${}^t m \cdot Y $ is
   $$2^{i+1}\cdot {}^t m_{i,i}\cdot y_{i, i+1}+2^{i+1}\cdot {}^t m_{i, i+1}\cdot y_{i+1, i+1}+ 2^{i+2}\cdot \widetilde{y_{i, i+1}}$$
   for some $\widetilde{y_{i, i+1}}$ and
   the $(i+1,i)$-block of ${}^t m \cdot Y $ is
   $$2^{i+1}\cdot {}^t m_{i+1, i}\cdot y_{i, i}+2^{i+1}\cdot {}^t m_{i+1, i+1}\cdot y_{i+1, i} + 2^{i+2}\cdot \widetilde{y_{i+1, i}}$$
   for some $\widetilde{y_{i+1, i}}$.
   From these, one can easily check that ${}^t m \cdot Y $ satisfies  condition d).
   \end{proof}

 Let $\underline{G}$ be the stabilizer of $q$ in $\underline{M}^{\ast}$. It is an affine group subscheme of $\underline{M}^{\ast}$, defined over $A$.
  Thus, we have the following theorem.
 \begin{Thm}
 The group scheme $\underline{G}$ is smooth, and $\underline{G}(R)=\mathrm{Aut}_R(L\otimes_A R,q\otimes_A R)$ for any \'{e}tale $A$-algebra $R$.
 \end{Thm}

\section{The special fiber}

In this section, we will determine the structure of the special fiber $\tilde{G}$ of $\underline{G}$ by observing the maximal reductive quotient and the component group. From this section to the end, the identity matrix is denoted by id.

    \subsection{The reductive quotient of the special fiber}

    Recall that $Z_i$ is the sublattice of $B_i$ such that $Z_i/2 A_i$ is the kernel of the quadratic form $\frac{1}{2^{i+1}}q$ mod 2 on $B_i/2 A_i$.
Let $\bar{V_i}=B_i/Z_i$ and $\bar{q}_i$ denote the nonsingular quadratic form $\frac{1}{2^{i+1}}q$ mod 2 on $\bar{V_i}$.
It is  obvious that each element of $\underline{G}(R)$ fixes $\bar{q}_i$ for every flat $A$-algebra $R$.
Based on this, we claim to have a morphism of algebraic groups
$$\varphi_i : \tilde{G} \rightarrow \mathrm{O}(\bar{V_i}, \bar{q_i})^{\mathrm{red}}$$
defined over $\kappa$, where $\mathrm{O}(\bar{V_i}, \bar{q_i})^{\mathrm{red}}$ is the reduced subgroup scheme of $\mathrm{O}(\bar{V_i}, \bar{q_i})$.
   Notice that  if the dimension of $\bar{V_i}$ is even and positive, then $\mathrm{O}(\bar{V_i}, \bar{q_i})^{\mathrm{red}} (= \mathrm{O}(\bar{V_i}, \bar{q_i}))$  is disconnected.
   If the dimension of $\bar{V_i}$ is odd, then
    $\mathrm{O}(\bar{V_i}, \bar{q_i})^{\mathrm{red}} (= \mathrm{SO}(\bar{V_i}, \bar{q_i}))$  is connected.

To prove the above claim,
let $R$ be an \'etale local $A$-algebra with $\kappa_R$ the residue field of $R$.
Since $\underline{G}$ is smooth over $A$, the map $\underline{G}(R)\rightarrow \tilde{G}(\kappa_R)$ is surjective by \textit{Hensel's lemma}.

Now, we choose an element $g\in \tilde{G}(\kappa_R)$ and its lifting $\tilde{g} \in \underline{G}(R)$.
Since $\tilde{g}$ induces an element of $\mathrm{O}(\bar{V_i}, \bar{q_i})^{\mathrm{red}}(\kappa_R)$,
 we have a map  from $\tilde{G}(\kappa_R)$ to $\mathrm{O}(\bar{V_i}, \bar{q_i})^{\mathrm{red}}(\kappa_R)$.
 It is easy to see that this map is well-defined, i.e. independent of a lifting $\tilde{g}$ of $g$.

 In order to show that this map is representable, we interpret it as matrices.
Recall that a matrix form of elements of $\tilde{G}(\kappa_R)$ is
$$\begin{pmatrix} 2^{max\{0,j-i\}}m_{i,j} \end{pmatrix},$$
where the diagonal block $m_{i,i}$ is $$\begin{pmatrix} s_i&2y_i\\ 2v_i&1+2z_i \end{pmatrix} \textit{or} \begin{pmatrix} s_i&r_i&2t_i\\ 2y_i&1+2x_i&2z_i\\ v_i&u_i&1+2w_i \end{pmatrix}$$ if $L_i$ is \textit{of type} $\textit{I}^o$ or \textit{of type} $\textit{I}^e$, respectively.

Let $g= \begin{pmatrix} 2^{max\{0,j-i\}}m_{i,j} \end{pmatrix}$.
Then $g$ maps to the following:

When $L_i$ is \textit{of type II},
$g$ maps to $ m_{i, i}$ (if $L_i$ is free) or
to $ \begin{pmatrix} m_{i, i}&0\\ \delta_{i-1}e_{i-1}\cdot m_{i-1, i}+\delta_{i+1}e_{i+1}\cdot m_{i+1, i}&1 \end{pmatrix}$ (if $L_i$ is bound).
Here, $
\delta_{j} = \left\{
  \begin{array}{l l}
  1    & \quad  \text{if $L_j$ is \textit{of type I}};\\
  0    &   \quad  \text{if $L_j$ is \textit{of type II}},
    \end{array} \right.
$
and $e_{j}=(0,\cdots, 0, 1)$ (resp. $e_j=(0,\cdots, 0, 1, 0)$) of size $1\times n_{j}$
if $L_{j}$ is \textit{of type} $\textit{I}^o$ (resp. \textit{of type} $\textit{I}^e$).

When $L_i$ is \textit{of type $I$},
$g$ maps to $ s_i$ if $L_i$ is \textit{free of type $I^o$} or \textit{free of type $I^e_2$}.
For the other cases with $L_i$  \textit{of type $I$}, namely if $L_i$ is \textit{free of type $I^e_1$}, \textit{bound of type $I^o$},  or \textit{bound of type $I^e$}, $g$ maps to $ \begin{pmatrix} s_i&0\\ \delta_{i}^{\prime}\sqrt{\widetilde{\gamma_i}}\cdot v_i +
(\delta_{i-1}e_{i-1}\cdot m_{i-1, i}+\delta_{i+1}e_{i+1}\cdot m_{i+1, i})\cdot\tilde{e_i}&1 \end{pmatrix}$.

Here, $\delta_{j}$ and $e_{j}$ are as explained above and $\sqrt{\widetilde{\gamma_i}}$ is as explained at the beginning of Appendix.
In addition,  
$\delta_{i}^{\prime} = \left\{
  \begin{array}{l l}
  1    & \quad  \text{if $L_i$ is \textit{of type $I^e$}};\\
  0    &   \quad  \text{if $L_i$ is \textit{of type $I^o$}},
    \end{array} \right.
$ and
$\tilde{e_i}=\begin{pmatrix} \mathrm{id}\\0 \end{pmatrix}$ of size $n_i\times (n_{i}-1)$ (resp. $n_i\times (n_{i}-2)$), where $\mathrm{id}$ is the identity matrix of size $(n_i-1)\times (n_{i}-1)$ (resp. $(n_i-2)\times (n_{i}-2)$) if $L_{i}$ is \textit{of type} $\textit{I}^o$ (resp. \textit{of type} $\textit{I}^e$).

This matrix interpretation induces the Hopf algebra morphism (polynomials of degree at most 1)
from the coordinate ring of $\mathrm{O}(\bar{V_i}, \bar{q_i})^{\mathrm{red}}$ to the coordinate ring of $\tilde{G}$,
which accordingly induces an algebraic group homomorphism $\varphi_i : \tilde{G} \rightarrow \mathrm{O}(\bar{V_i}, \bar{q_i})^{\mathrm{red}}$ such that
the group homomorphism induced by $\varphi_i$ at the level of $\kappa_R$-points is as given above.

Since $\tilde{G}$ is smooth over $\kappa$, the set of $\kappa_R$-points of $\tilde{G}$ for all finite  extensions $\kappa_R/\kappa$
is dense in $\tilde{G}$ by Corollary 13 of Section 2.2 in \cite{BLR}. 
Therefore, $\varphi_i$ is uniquely determined by the map constructed above at the level of $\kappa_R$-points.

\begin{Thm}
The morphism $\varphi$ defined by
 $$\varphi=\prod_i \varphi_i : \tilde{G} ~ \longrightarrow  ~\prod_i \mathrm{O}(\bar{V_i}, \bar{q_i})^{\mathrm{red}}$$
 is surjective.
\end{Thm}

\begin{proof}
 Assume that dimension of $\tilde{G} =$
 dimension of $\mathrm{Ker~}\varphi$ + $\sum_i$ (dimension of $\mathrm{O}(\bar{V_i}, \bar{q_i})^{\mathrm{red}}$).
Thus $\mathrm{Im~}\varphi$ contains the identity component of $\prod_i \mathrm{O}(\bar{V_i}, \bar{q_i})^{\mathrm{red}}$.
Here $\mathrm{Ker~}\varphi$ denotes the kernel of $\varphi$ and $\mathrm{Im~}\varphi$ denotes the image of $\varphi.$

Recall that a matrix form of elements of $\tilde{G}(R)$ for a $\kappa$-algebra $R$ is
$$m= \begin{pmatrix} 2^{max\{0,j-i\}}m_{i,j} \end{pmatrix},$$
where the diagonal block $m_{i,i}$ is $$\begin{pmatrix} s_i&2y_i\\ 2v_i&1+2z_i \end{pmatrix} \textit{or} \begin{pmatrix} s_i&r_i&2t_i\\ 2y_i&1+2x_i&2z_i\\ v_i&u_i&1+2w_i \end{pmatrix}$$ if $L_i$ is \textit{of type} $\textit{I}^o$ or \textit{of type} $\textit{I}^e$, respectively.

 Let $\mathcal{H}$ be the set of $i$'s such that $\mathrm{O}(\bar{V_i}, \bar{q_i})^{\mathrm{red}}$ is disconnected.
Notice that $\mathrm{O}(\bar{V_i}, \bar{q_i})^{\mathrm{red}}$ is disconnected exactly when $L_i$ is \textit{free} \textit{of type II}, \textit{free} \textit{of type} $\textit{I}^o$,
or \textit{free} \textit{of type} $\textit{I}^e_2$.
For such a lattice $L_i$, we define the closed subgroup scheme $H_i$ of $\tilde{G}$ as follows:
\begin{itemize}
\item If $L_i$ is \textit{free} \textit{of type II}, then $H_i$ is defined by the equations $m_{j,k}=0$ if $ j\neq k$, and $m_{j,j}=\mathrm{id}$ if $j \neq i$.
\item If $L_i$ is \textit{free} \textit{of type} $\textit{I}^o$, then  $H_i$ is defined by the equations $m_{j,k}=0 $ if $ j\neq k$, $m_{j,j}=\mathrm{id}$ if $j \neq i$,
and $y_i=0, v_i=0, z_i=0$.
\item If $L_i$ is \textit{free} \textit{of type} $\textit{I}^e_2$, then  $H_i$ is defined by the equations $m_{j,k}=0 $ if $ j\neq k$, $m_{j,j}=\mathrm{id}$ if $j \neq i$,
and $r_i=0, t_i=0, y_i=0, x_i=0, z_i=0, v_i=0, u_i=0, w_i=0$.
\end{itemize}
Then $\varphi_i$ induces an isomorphism between $H_i$ and
$\mathrm{O}(\bar{V_i}, \bar{q_i})^{\mathrm{red}}$.
We consider the morphism
\[\prod_{i\in\mathcal{H}} H_i \longrightarrow \tilde{G},\]
$(h_i)_{i\in\mathcal{H}}\mapsto\prod_{i\in\mathcal{H}}h_i$.
Note that   $H_i$ and $H_j$ commute with each other in the sense that $h_i\cdot h_j=h_j\cdot h_i$ for all $i \neq j$, where $h_i\in H_i(R)$ and $ h_j\in H_j(R)$
for a $\kappa$-algebra $R$.
Based on this, the above morphism becomes a group homomorphism.
In addition, the fact that
$H_i\cap H_j=0$ for all $i\neq j$
implies that this morphism is injective.
Thus the product $\prod_{i\in\mathcal{H}} H_i$ is embedded into $\tilde{G}$ as a closed subgroup scheme.
Since $\varphi_i|_{H_j}$ is trivial for $i\neq j$, the morphism
$$\prod_{i\in\mathcal{H}}\varphi_i : \prod_{i\in\mathcal{H}}H_i \rightarrow \prod_{i\in\mathcal{H}} \mathrm{O}(\bar{V_i}, \bar{q_i})$$
is an isomorphism.
Therefore, $\varphi$ is surjective.
Now it suffices to establish the assumption made at the beginning of the proof, which is the next lemma.
\end{proof}

\begin{Lem}

$\mathrm{Ker~}\varphi $ is isomorphic to $ \textbf{A}^{l}\times (\mathbb{Z}/2\mathbb{Z})^{\alpha+\beta}$ as $\kappa$-varieties,
 where $\textbf{A}^{l}$ is an affine space of the dimension $l$.
 Here,
 \begin{itemize}
 \item $\alpha$ is the number of $i$'s such that $L_i$ is \textit{free} \textit{of type} $\textit{I}^e_1$.
 \item  $\beta$ is the size of the set of $j$'s such that $L_j$ is \textit{of type I} and $L_{j+2}$ is \textit{of type II}.
\item  $l$ is such that $l + \sum_i$ (dimension of $\mathrm{O}(\bar{V_i}, \bar{q_i})^{\mathrm{red}}$) $=$  dimension of $\tilde{G}$.
\end{itemize}

\end{Lem}

The proof is postponed to the Appendix.

\begin{Rmk}
 We describe Im $\varphi_i$ as follows.
     \[
      \begin{array}{c|c}
      \mathrm{Type~of~lattice~}  L_i & \mathrm{Im~}  \varphi_i \\
      \hline
      \textit{I}^o,\ \  \mathrm{\textit{free}} & \mathrm{O}(n_i-1, \bar{q_i})\\
      \textit{I}^e_1,\ \  \mathrm{\textit{free}} &\mathrm{SO}(n_i-1, \bar{q_i})\\
      \textit{I}^e_2,\ \  \mathrm{\textit{free}} &\mathrm{O}(n_i-2, \bar{q_i})\\
      \textit{II},\ \  \mathrm{\textit{free}} &\mathrm{O}(n_i, \bar{q_i})\\
      \textit{I}^o,\ \  \mathrm{\textit{bound}} &\mathrm{SO}(n_i, \bar{q_i})\\
      \textit{I}^e,\ \  \mathrm{\textit{bound}} &\mathrm{SO}(n_i-1, \bar{q_i})\\
      \textit{II},\ \  \mathrm{\textit{bound}} &\mathrm{SO}(n_i+1, \bar{q_i})\\
      \end{array}
    \]
\end{Rmk}

     \subsection{The first construction of component groups}

The purpose of this subsection and the next subsection is to define the surjective morphism from $\tilde{G}$ to $(\mathbb{Z}/2\mathbb{Z})^{\alpha +\beta}$,
where $\alpha$ and $\beta$ are defined in Lemma 4.2.\\

   We first define two constant group schemes $E_i$ and $F_i$.
We consider the closed subgroup scheme $\tilde{H}_i$ of $\tilde{G}$ defined by equations
$$m_{i,j}=0 \mathrm{~and~} m_{j,j}=\mathrm{id} \text{ for all } j\neq i.$$
Then $\tilde{H}_i$ is isomorphic to the special fiber of the smooth affine group scheme associated to the lattice $L_i$.

If $L_i$ is \textit{of type} $\textit{I}^e$, then a matrix form of elements of $\tilde{H}_i$ is
\[\begin{pmatrix} s&r&2t\\ 2y&1+2x&2z\\ v&u&1+2w \end{pmatrix}.\]
The subgroup scheme $E_i$ is defined by the following equations:
   $$s=\mathrm{id}, r=0, t=0, y=0, v=0, z=0, \text{ and } w=0.$$
   The equations defining the subgroup scheme $F_i$ is as follows:
   $$s=\mathrm{id}, r=0, t=0, y=0, v=0, u=0, \text{ and }  w=0.$$
   Obviously, these are closed subgroup schemes of $\tilde{G}$.
   If $L_i$ is \textit{of type} $\textit{I}^e$, then  $F_i$ is isomorphic to $\mathbb{Z}/2\mathbb{Z}$.
In particular,
   if $L_i$ is \textit{free of type} $\textit{I}^e_1$, then $E_i$   is  isomorphic to $\mathbb{Z}/2\mathbb{Z}$ as well.

If $L_i$ is \textit{of type} $\textit{I}^o$, then a matrix form of elements of $\tilde{H}_i$ is
\[\begin{pmatrix} s_i&2y_i\\ 2v_i&1+2z_i \end{pmatrix}.\]
The closed subgroup scheme $F_i$ of $\tilde{H}_i$ is defined by equations
$$s_i=\mathrm{id}, y_i=0,  \text{ and } v_i=0$$
 and it is isomorphic to $\mathbb{Z}/2\mathbb{Z}$.\\

    We now construct the morphism $\xi_i : \tilde{G} \longrightarrow \mathbb{Z}/2\mathbb{Z}$.
    Assume that $L_i$ is \textit{free} \textit{of type} $\textit{I}^e_1$.
    Consider the linear form
    $$\frac{1}{2^i}\langle -,e \rangle \mathrm{~mod~}2 \mathrm{~on~} A_i/X_i,$$
     where $e$ is the vector in $A_i/X_i$ as defined in Section 2.3.
    This linear form is fixed by elements of $\underline{G}(R)$ for a flat $A$-algebra $R$,
    because the vector $e$ is fixed by elements of $\underline{G}(R)$.
    We choose an arbitrary linear form $l$ on $A_i$ such that
    $$l ~\mathrm{mod~} 2 = \frac{1}{2^i}\langle -,e \rangle ~\mathrm{mod~} 2 \mathrm{~on~} A_i/X_i.$$
    Define the quadratic form $q^{\prime}$ by $l^2$. Notice that the norm of the quadratic lattice $(A_i, \frac{1}{2^i}q+q^{\prime})$ is the prime ideal $(2)$.
 If we consider the quadratic form $\frac{1}{2}(\frac{1}{2^i}q+q^{\prime} ) ~\mathrm{mod~} 2$ defined over the $\kappa$-vector space $A_i/X_i$,
 then it is stabilized by elements of $\underline{G}(R)$ for a flat $A$-algebra $R$.
  It is obvious that this quadratic form defined over $A_i/X_i$ is nonsingular and independent of the choice of $l$.
 Thus we have a morphism of algebraic groups
 $$\tilde{G} \longrightarrow \mathrm{O}(A_i/X_i, \frac{1}{2}(\frac{1}{2^i}q+q^{\prime} ))$$
 defined over $\kappa$.
 As  matrices, if we express $g=\begin{pmatrix} 2^{max\{0,j-i\}}m_{i,j} \end{pmatrix}\in \tilde{G}(R)$ with\\
 $m_{i,i}= \begin{pmatrix} s_i&r_i&2t_i\\ 2y_i&1+2x_i&2z_i\\ v_i&u_i&1+2w_i \end{pmatrix}$ for a $\kappa$-algebra $R$,
 then $g$ maps to  $\begin{pmatrix} s_i&r_i&0\\ 0&1&0\\ v_i&u_i&1 \end{pmatrix}$.

 Note that the dimension of the $\kappa$-vector space $A_i/X_i$ is the same as that of $L_i/2L_i$, which is  the even integer $n_i$.
 On the other hand, there is a surjective morphism from this orthogonal group onto $\mathbb{Z}/2\mathbb{Z}$, namely the Dickson invariant, since the dimension of the $\kappa$-vector space $A_i/X_i$ is even.
 We define
 $$\xi_i : \tilde{G} \longrightarrow \mathbb{Z}/2\mathbb{Z}$$
to be the composition of the Dickson invariant and the preceding morphism.

\begin{Rmk}
We describe the Dickson invariant as the determinant morphism of a smooth affine group scheme.
Let $(\bar{V}, \bar{q})$ be a nonsingular quadratic space, where $\bar{V}$ is a $\kappa$-vector space of even dimension.
Then we can choose a unimodular lattice $(L, q)$  \textit{of type II}
such that $(L/2L, \frac{1}{2}q ~\textit{mod 2})=(\bar{V}, \bar{q})$.
If $\mu_{n, A}$ is the group scheme of $n$-th roots of unity defined over $A$, then
the determinant morphism gives a morphism of  group schemes
$$\textit{det} : \underline{G}^{\prime} \rightarrow \mu_{2, A}$$
    defined over $A$.
    Here, $\underline{G}^{\prime}$ is a naive integral model such that
    $\underline{G}^{\prime}(R)=\mathrm{Aut}_{R}(L\otimes_AR, q\otimes_AR)$ for every commutative $A$-algebra $R$.
    Moreover, we can regard $\mu_{2, A}$ as a naive integral model of the orthogonal group associated to a quadratic lattice of rank 1.
Then it is easily seen that the smooth affine group scheme associated to this quadratic lattice of rank 1
 is $\mathbb{Z}/2\mathbb{Z}$, by observing the equation defining it.

   Based on the above, the morphism $det$ induces
     a morphism of group schemes from $\underline{G}$ to $\mathbb{Z}/2\mathbb{Z}$ defined over $A$ by functoriality of smooth integral models
     and this morphism gives the morphism
     $$\widetilde{det} : \tilde{G} \rightarrow \mathbb{Z}/2\mathbb{Z}$$
     defined over $\kappa$.
Furthermore, it is easily seen that $\widetilde{det}$ is surjective.
In fact, the kernel of $\widetilde{det}$ is the identity component.
To see this, we observe that  the morphism
 \[\varphi : \tilde{G} ~ \longrightarrow   \mathrm{O}(L/2L, \frac{1}{2}q ~\textit{mod 2}) (=\mathrm{O}(\bar{V}, \bar{q}))\]
 is an isomorphism by Theorem 4.1 and Lemma 4.2 and so
 $\tilde{G}$ has two connected components.

 On the other hand, the Dickson invariant gives a surjective morphism from $\tilde{G} (\cong \mathrm{O}(\bar{V}, \bar{q}))$
 to $\mathbb{Z}/2\mathbb{Z}$ and its kernel is the identity component as well.
    Therefore, the Dickson invariant is the same as the morphism $\widetilde{det}$.

\end{Rmk}

\begin{Lem}
The restricted morphism $$\xi_i|_{E_i} : E_i \longrightarrow \mathbb{Z}/2\mathbb{Z}$$ is an isomorphism.
Recall that $E_i$ is defined at the beginning of this subsection.
\end{Lem}

\begin{proof}
If we consider the closed subgroup scheme of $\tilde{H}_i$ defined by equations
$$s=\mathrm{id}, r=0, t=0, y=0, \text{ and } v=0,$$
then this group scheme is isomorphic to the special fiber of the smooth affine group scheme associated to a lattice \textit{free of type} $\textit{I}^e_1$ with rank 2.
Since $E_i$ is a subgroup scheme of this group scheme, we may and do assume that $n_i=2$.

Based on Remark 4.4,
we describe the  morphism $\widetilde{det}$ associate to the orthogonal group $\mathrm{O}(A_i/X_i, \frac{1}{2}(\frac{1}{2^i}q+q^{\prime} ))$
explicitly.
Choose a lattice of rank 2 with a Gram matrix $\begin{pmatrix} 2&1\\ 1&2\gamma_i \end{pmatrix}$.
Here, $\gamma_i$ is a unit in $A$ such that $L_i=A(1, 2\gamma_i)$.
Since this lattice is unimodular \textit{of type II},
a matrix form of a flat $A$-algebra point of the associated smooth integral model is $\begin{pmatrix} x&y\\ u&z \end{pmatrix}$. In other words, there are no congruence conditions.
This matrix satisfies three equations:
$$
x^2+ux+\gamma_i\cdot u^2=1, ~~~
2xy+xz+uy+2\gamma_i\cdot zu=1, ~~~
y^2+yz+\gamma_i\cdot z^2=\gamma_i.
$$
The determinant of this matrix is $xz-uy=1-2(uy+xy+\gamma_i\cdot zu)$.
We also express a $\kappa$-algebra point of the smooth integral model as the matrix $\begin{pmatrix} x&y\\ u&z \end{pmatrix}$.
Then
 $\widetilde{det}$ maps $\begin{pmatrix} x&y\\ u&z \end{pmatrix}$  to  $uy+xy+\widetilde{\gamma_i}\cdot zu$.
Here, $\widetilde{\gamma_i} (\neq 0)$ is the image of $\gamma_i$ in the residue field $\kappa$.

For $g\in E_i(R)$, $\xi_i|_{E_i}(g)=\widetilde{det}(\begin{pmatrix} 1&0\\ u_i&1 \end{pmatrix})=\widetilde{\gamma_i}\cdot u_i$.
Therefore, $\xi_i|_{E_i}$ is an isomorphism.
\end{proof}

Combining all morphism $\xi_i$'s, we have the following theorem:
\begin{Thm}
 The morphism $\xi=\prod_i\xi_i : \tilde{G} \rightarrow (\mathbb{Z}/2\mathbb{Z})^{\alpha}$ is surjective.
  Here, $\alpha$ is the number of $i$'s such that $L_i$ is \textit{free} \textit{of type} $\textit{I}^e_1$.

\end{Thm}

\begin{proof}
Define the scheme $E$ to be the product of $E_i$'s such that $L_i$ is \textit{free} \textit{of type} $\textit{I}^e_1$.
Notice that $E_i$ and $E_j$  commute with each other in the sense that $e_i\cdot e_j=e_j\cdot e_i$ for all $i \neq j$, where $e_i\in E_i $ and $ e_j\in E_j$
and $L_i$ and $L_j$ are \textit{free} \textit{of type} $I^e_1$, and that $E_i\cap E_j=0$.
Thus $E$ can be embedded into $\tilde{G}$ as a closed subgroup scheme. In addition, it is obvious that $\xi_i|_{E_j}$ is trivial for $i \neq j$.
Therefore $\xi$ induces an isomorphism of algebraic groups from $E$ to $(\mathbb{Z}/2\mathbb{Z})^{\alpha}$ defined over $\kappa$.
This completes the proof.
\end{proof}

     \subsection{The second construction of component groups}

In this subsection, we will construct the morphism $\psi$ from $\tilde{G}$ to $(\mathbb{Z}/2\mathbb{Z})^{\beta}$.
We begin by defining several lattices.
    \begin{Def}
     We define the lattice $L^1$ which is the sublattice of $L$ such that $L^1/2L$ is the kernel of the symmetric bilinear form $\langle-,-\rangle$ mod 2 on $L/2L$.
      Similarly we define the lattice $L^i$ which is the sublattice of $L^{i-1}$ such that $L^i/2L^{i-1}$ is the kernel of
      the symmetric bilinear form $\frac{1}{2^{i-1}}\langle-,-\rangle$ mod 2 on $L^{i-1}/2L^{i-1}$ for all $i\geq 1$.
     For simplicity, put $$L^{0}=L=\bigoplus_{i\geq 0}L_i, 0\leq i < N.$$
     The description of $L^i$ is
      $$L^{2m}=2^m(L_0\oplus L_1)\oplus2^{m-1}(L_2\oplus L_3)\oplus \cdots \oplus 2(L_{2m-2}\oplus L_{2m-1})\oplus \bigoplus_{i\geq 2m}L_i$$
     and
     $$L^{2m-1}=2^{m}L_0 \oplus2^{m-1}(L_1\oplus L_2)\oplus \cdots \oplus 2(L_{2m-3}\oplus L_{2m-2}) \oplus \bigoplus_{i\geq 2m-1}L_i.$$
\end{Def}

     We choose a Jordan splitting for the quadratic lattice $(L^{2m}, \frac{1}{2^{2m}}q)$ as follows:
     $$L^{2m}=\bigoplus_{i \geq 0} M_i,$$ where
     $$M_0=2^mL_0\oplus2^{m-1}L_2\oplus \cdots \oplus 2L_{2m-2}\oplus L_{2m},$$
     $$M_1=2^mL_1\oplus2^{m-1}L_3\oplus \cdots \oplus 2L_{2m-1}\oplus L_{2m+1}$$
     $$\mathrm{and}~ M_k=L_{2m+k} \mathrm{~if~} k\geq 2.$$
     Here, $M_i$ is \textit{modular} and $S(M_i)=(2^i)$.
     For the quadratic lattice $(L^{2m-1}, \frac{1}{2^{2m-1}}q)$, a chosen Jordan splitting is as follows:
     $$L^{2m-1}=\bigoplus_{i \geq 0} M_i,$$ where
     $$M_0=2^{m-1}L_1\oplus2^{m-2}L_3\oplus \cdots \oplus 2L_{2m-3}\oplus L_{2m-1},$$
     $$M_1=2^{m}L_0\oplus2^{m-1}L_2\oplus \cdots \oplus 2L_{2m-2}\oplus L_{2m}$$
     $$\mathrm{and}~ M_{k}=L_{2m-1+k} \mathrm{~if~} k\geq 2.$$

\begin{Def}
We define $C(L)$ to be the sublattice of $L$ such that $$C(L)=\{x\in L \mid \langle x,y\rangle \in (2) \ \ \mathrm{for}\ \ \mathrm{all}\ \ y \in B(L)\}.\\$$
\end{Def}

    We  choose any integer $j$ such that $L_{j}$ is \textit{of type} $\textit{I}$ and $L_{j+2}$ is \textit{of type} $\textit{II}$.
    We stress that $M_0$ is \textit{of type} $\textit{I}$ and $M_2=L_{j+2}$ is \textit{of type} $\textit{II}$.
     We choose a basis $(\langle e_i\rangle, e)$ (resp. $(\langle e_i\rangle, a, e)$) for $M_0$ based on Theorem 2.4 when the rank of $M_0$ is odd (resp. even).
     Then $B(L^{j})$ is spanned by $$(\langle e_i\rangle, 2e) ~(resp.~ (\langle e_i\rangle, 2a, e)) \text{~and~}  M_1 \oplus (\bigoplus_{i\geq 2} M_i)$$
      and $C(L^{j})$ is spanned by $$(\langle 2e_i\rangle, e)~ (resp. ~(\langle 2e_i\rangle, 2a, e)) \text{~and~}  M_1 \oplus (\bigoplus_{i\geq 2} M_i).$$\\

We now construct the morphism $\psi_j : \tilde{G} \rightarrow \mathbb{Z}/2\mathbb{Z}$ as follows
(There are 2 cases depending on whether $M_0$ is \textit{of type} $\textit{I}^e$ or \textit{of type} $\textit{I}^o$.):

\begin{itemize}
\item[(1)] Firstly, we assume that $M_0$ is \textit{of type} $\textit{I}^e$.
We choose a Jordan splitting for the quadratic lattice $(C(L^j), \frac{1}{2^{j+1}}q)$ as follows:
$$C(L^j)=\bigoplus_{i \geq 0} M_i^{\prime}.$$
Notice that $M_1^{\prime}$ is \textit{of type II} so that $M_0^{\prime}$ is \textit{free}.
Let $G_j$ denote the special fiber of the smooth affine group scheme associated to the quadratic lattice $(C(L^j), \frac{1}{2^{j+1}}q)$.
We now have a morphism from $\tilde{G}$ to $G_j$.

 If $M_0^{\prime}$ is \textit{of type II}, \textit{of type} $\textit{I}^o$ or \textit{of type} $\textit{I}^e_2$,
 then we have a morphism from $G_j$ to the even orthogonal group associated to $M_0^{\prime}$
 as explained in Section 4.1.
Thus, the Dickson invariant of this orthogonal group induces the morphism
$$\psi_j : \tilde{G} \longrightarrow \mathbb{Z}/2\mathbb{Z}.$$
If $M_0^{\prime}$ is  \textit{of type} $\textit{I}^e_1$, then  we have a morphism from $G_j$ to $\mathbb{Z}/2\mathbb{Z}$ associated to $M_0^{\prime}$ as explained in Section 4.2.
It induces the morphism
$$\psi_j : \tilde{G} \longrightarrow \mathbb{Z}/2\mathbb{Z}.$$

\item[(2)] We next assume that $M_0$ is \textit{of type} $\textit{I}^o$.
We choose a Jordan splitting for the quadratic lattice $(C(L^j), \frac{1}{2^{j}}q)$ as follows:
$$C(L^j)=\bigoplus_{i \geq 0} M_i^{\prime}.$$
Notice that the rank of the unimodular lattice $M_0^{\prime}$ is 1 and the lattice $M_2^{\prime}$ is \textit{of type II}.
If $G_j$ denotes the special fiber of the smooth affine group scheme associated to the quadratic lattice $(C(L^j), \frac{1}{2^{j}}q)$,
we  have a morphism from $\tilde{G}$ to $G_j$.

We now consider the new quadratic lattice $M_0^{\prime}\oplus C(L^j)$.
The smooth affine group scheme associated to the quadratic lattice $(C(L^j), \frac{1}{2^{j}}q)$ can be embedded into
the smooth affine group scheme associated to the quadratic lattice $M_0^{\prime}\oplus C(L^j)$ as a closed subgroup scheme.
Thus the special fiber $G_j$ of the former group scheme is embedded into the special fiber of the latter group scheme.
Since the unimodular lattice $M_0^{\prime}\oplus M_0^{\prime}$ is \textit{of type} $\textit{I}^e$,
where $(M_0^{\prime}\oplus M_0^{\prime})\oplus \bigoplus_{i \geq 1} M_i^{\prime}$ is a Jordan splitting of the quadratic lattice $M_0^{\prime}\oplus C(L^j)$,
we have a morphism from the special fiber of the latter group scheme to $\mathbb{Z}/2\mathbb{Z}$ as constructed in the first case.
It induces the morphism
$$\psi_j : \tilde{G} \longrightarrow \mathbb{Z}/2\mathbb{Z}.$$

\item[(3)] Combining all cases, we have the morphism $$\psi=\prod_j \psi_j : \tilde{G} \longrightarrow (\mathbb{Z}/2\mathbb{Z})^{\beta},$$
where $\beta$ is the size of the set of $j$'s such that $L_j$ is \textit{of type I} and  $L_{j+2}$ is \textit{of type II}.
\end{itemize}

    \begin{Rmk}
     There is another description for $\beta$. We consider the type sequence $\{a_i\}$ such that $a_i$ is $I$ or $II$ according to the parity type of $L_i$.
     Define two sequences $b_m=a_{2m+1}$ and $c_m=a_{2m}$. Then we  take maximal consecutive terms consisting of $I$ in each $b_m$ or $c_m$.
     The set consisting of these terms  is finite and its size is $\beta$.
     For example, if $$\{a_n\}_{n\geq 0}= \{I ~I~ I~ II~ I~ I~ II~ I~ II~ I~ I\},$$
     then $$\{b_m\}_{m\geq 0} = \{I~ II~ I~ I~ I\}\text{ and }  \{c_m\}_{m\geq 0} = \{I~ I~ I~ II~ II~ I\}.$$
     Hence $\beta$ is $2+2=4$.
     \end{Rmk}

We now have the following result.
\begin{Thm}
 The morphism $$\psi=\prod_j \psi_j : \tilde{G} \longrightarrow (\mathbb{Z}/2\mathbb{Z})^{\beta}$$ is surjective.

Moreover, the morphism $$\varphi \times \xi \times \psi : \tilde{G} \rightarrow \prod_i \mathrm{O}(\bar{V_i}, \bar{q_i})^{\mathrm{red}} \times (\mathbb{Z}/2\mathbb{Z})^{\alpha+\beta}$$
is also surjective.

\end{Thm}

\begin{proof}
We first show that $\psi_j$ is surjective. Note that for such a $j$, $L_j$ is \textit{of type I} and $L_{j+2}$ is \textit{of type II}.
Recall that we have defined the closed subgroup scheme $F_j$ of $\tilde{G}$ at the beginning of Section 4.2 and it is isomorphic to $\mathbb{Z}/2\mathbb{Z}$.
Now it suffices to show that $\psi_j|_{F_j}$ is an isomorphism and its proof is similar to that of Lemma 4.5 so we may skip.

 Surjectivity  of $\psi$ is similar to Theorem 4.6.
Notice that $F_i$ and $F_j$ commute with each other for all $i\neq j$, where $L_i$ and $L_j$ (resp. $L_{i+2}$ and $L_{j+2}$) are \textit{of type I}
(resp. \textit{of type II}), and that $F_i\cap F_j=0$.
Thus the product $F=\prod_j F_j$ is embedded into $\tilde{G}$ as a closed subgroup scheme.
In addition, it is obvious that $\psi_i|_{F_j}$ is trivial for all $i<j$.
Hence the morphism $\psi$ induces an isomorphism of algebraic groups from $F$ to $(\mathbb{Z}/2\mathbb{Z})^{\beta}$ defined over $\kappa$.
This shows  surjectivity  of the morphism $\psi$.

For  surjectivity of $\varphi \times \xi \times \psi$, it suffices to show that $\xi \times \psi|_{\mathrm{Ker~}\varphi}$ is surjective onto $(\mathbb{Z}/2\mathbb{Z})^{\alpha+\beta}$.
Since the morphism $\varphi$ vanishes on $E$ and $F$, the two schemes $E$ and $F$ are subschemes of $\mathrm{Ker~}\varphi$.
Notice that the intersection of $E$ and $F$ as subgroup schemes of $\mathrm{Ker~}\varphi$ is trivial.
This fact implies that the product $E\times F$ is embedded into $\mathrm{Ker~}\varphi$ as $\kappa$-schemes.
Notice that $E$ and $F$ may not commute with each other so  $E \times F$ may not inherit subgroup scheme structure of $\mathrm{Ker~}\varphi$.
It is easily seen from the construction of $E$ and $F$ that the restricted morphisms $\xi|_{F}$ and $\psi|_{E}$ are trivial.
Therefore, $\xi \times \psi$ induces an isomorphism from $E\times F$ to $(\mathbb{Z}/2\mathbb{Z})^{\alpha+\beta}$ as $\kappa$-schemes.
This completes the proof.
\end{proof}

     \subsection{The maximal reductive quotient of $\tilde{G}$}

  Let $\tilde{M}$ be the special fiber of $\underline{M}^{\ast}$. Let
  $$\tilde{M_i}=\mathrm{GL}_{\kappa}(B_i/Y_i).$$
 For any $\kappa$-algebra $R$, let $m=\begin{pmatrix} 2^{max\{0,j-i\}}m_{i,j} \end{pmatrix} \in \tilde{M}(R)$.
 Recall that $s_i$ is a block of $m_{i,i}$ if $L_i$ is \textit{of type I}, as explained in Section 3.1.
 Then $s_i \in \tilde{M_i}(R)$. If $L_i$ is \textit{of type II}, then $m_{i,i} \in \tilde{M_i}(R)$.
 Therefore, we have a surjective morphism of algebraic groups
 $$r : \tilde{M} \longrightarrow \prod\tilde{M}_i,$$
 defined over $\kappa$. We now have the following lemma:

 \begin{Lem}
The kernel of $r$ is the unipotent radical $\tilde{M}^+$ of $\tilde{M}$, and $\prod\tilde{M}_i$ is the maximal reductive quotient of $\tilde{M}$.

\end{Lem}

     We finally have the structural theorem for the algebraic group $\tilde{G}$.

     \begin{Thm}
     The morphism $$\varphi \times \xi \times \psi : \tilde{G} \longrightarrow  \prod_i \mathrm{O}(\bar{V_i}, \bar{q_i})^{\mathrm{red}} \times (\mathbb{Z}/2\mathbb{Z})^{\alpha+\beta}$$
      is surjective and the kernel is  unipotent and connected.
     Consequently, $\prod_i \mathrm{O}(\bar{V_i}, \bar{q_i})^{\mathrm{red}} \times (\mathbb{Z}/2\mathbb{Z})^{\alpha+\beta}$ is the maximal reductive quotient.
     Here, $\mathrm{O}(\bar{V_i}, \bar{q_i})^{\mathrm{red}}$ is explained in Section 4.1 (especially, Remark 4.3), and
     $\alpha$ and $\beta$ are defined in Lemma 4.2.
     \end{Thm}

\begin{proof}
     We only need to prove that the kernel is unipotent and connected.
     Since the kernel of $\varphi$ is a closed subgroup scheme of the unipotent group $\tilde{M}^+$,
     it suffices to show that the kernel of $\varphi \times \xi \times \psi$ is connected.
    Equivalently, it suffices to show that the kernel of the restricted morphism $\xi \times \psi|_{\mathrm{Ker~}\varphi}$ is connected.
     From Lemma 4.2,
     $\mathrm{Ker~}\varphi \cong \textbf{A}^{l}\times (\mathbb{Z}/2\mathbb{Z})^{\alpha+\beta}$ as $\kappa$-varieties.
 Since the restricted morphism $\xi \times \psi|_{\mathrm{Ker~}\varphi}$ is surjective onto $(\mathbb{Z}/2\mathbb{Z})^{\alpha+\beta}$,
 we complete the proof by counting the number of connected components.
\end{proof}

    \begin{Rmk}
    Recall that $\alpha$ is the number of $i$'s such that $L_i$ is \textit{free} \textit{of type} $\textit{I}^e_1$.
    For such a lattice, Im$~\varphi_i$ is $\mathrm{SO}(n_i-1)$ which is connected.
    If a lattice is \textit{free} with nontrivial $\bar{V_i}$
    but not \textit{of type} $\textit{I}^e_1$, then Im$~\varphi_i$ is disconnected with two connected components.
    Therefore the maximal reductive quotient $$\prod \mathrm{O}(\bar{V_i}, \bar{q_i})^{\mathrm{red}} \times (\mathbb{Z}/2\mathbb{Z})^{\alpha+\beta}$$
    is isomorphic to
    $$\prod \mathrm{SO}(\bar{V_i}, \bar{q_i})\times (\mathbb{Z}/2\mathbb{Z})^{\alpha^{\prime}+\beta},$$
    as $\kappa$-varieties, where $\alpha^{\prime}$ is the number of $i$'s such that $L_i$ is \textit{free} with nontrivial $\bar{V_i}$.
    \end{Rmk}

\section{Comparison of volume forms and final formulas}

This section is based on Section 7 of \cite{GY}.
In  the construction of Section 3.2 of \cite{GY}, pick $\omega^{\prime}_M$ and $\omega^{\prime}_Q$ to be such that
$$\int_{\underline{M}(A)}|\omega^{\prime}_M|=1 \mathrm{~and~}  \int_{\underline{Q}(A)}|\omega^{\prime}_Q|=1. $$
Put $\omega^{\mathrm{can}}=\omega^{\prime}_M/\rho^{\ast}\omega^{\prime}_Q$.
By Theorem 3.5, we have an exact sequence of locally free sheaves on $\underline{M}^{\ast}$:
\[ 0\longrightarrow \rho^{\ast}\Omega_{\underline{Q}/A} \longrightarrow \Omega_{\underline{M}^{\ast}/A}
\longrightarrow \Omega_{\underline{M}^{\ast}/\underline{Q}} \longrightarrow 0. \]
It follows that $\omega^{\mathrm{can}}$ is of the type discussed in Section 3 of \cite{GY}.

\begin{Lem} Let $\pi$ be a uniformizer of $A$. Then
  $$\omega_M=\pi^{N_M}\omega_M^{\prime}, \ \ \ \  N_M=\sum_{L_i:\textit{type I}}(2n_i-1)+\sum_{i<j}(j-i)\cdot n_i\cdot n_j+2b,$$
  $$\omega_Q=\pi^{N_Q}\omega_Q^{\prime},  \ \ \ \ N_Q=\sum_{L_i:\textit{type I}}2n_i+\sum_{i<j}j\cdot n_i\cdot n_j+\sum_i d_i+b+c,$$
  $$\omega^{\mathrm{ld}}=\pi^{N_M-N_Q}\omega^{\mathrm{can}}.$$
  Here
  \begin{itemize}
  \item $b$ is the total number of pairs of adjacent constituents $L_i$ and $L_{i+1}$ that are both \textit{of type I}.
  $($b is denoted by $n(\textit{I},\textit{I})$ in \cite{CS}.$)$
  \item $c$ is the sum of dimensions of all nonempty Jordan constituents $L_i$'s that are \textit{of type} $\textit{II}$.
  $($c is denoted by  $n(\textit{II})$  in \cite{CS}.$)$
  \item $d_i=i\cdot n_i\cdot (n_i+1)/2$.
  \end{itemize}
\end{Lem}

\begin{Thm}
Let $f$ be the cardinality of $\kappa$.
The local density of ($L,q$) is

$$\beta_L=\frac{1}{[\mathrm{O}(V, q):\mathrm{SO}(V, q)]}f^N \cdot f^{-\mathrm{dim~} \mathrm{O}(V, q)} \sharp\tilde{G}(\kappa),$$

where $N=N_Q-N_M=t+\sum_{i<j} i\cdot n_i\cdot n_j+\sum_i d_i-b+c$, $t$ = the total number of $L_i$'s that are \textit{of type I}.
Here, $\sharp\tilde{G}(\kappa)$ can be computed based on Remark 5.3.(1) and Theorem 4.12.
\end{Thm}

\begin{Rmk}

\begin{enumerate}
\item In the above local density formula, $\sharp\tilde{G}(\kappa)$ is  computed as follows.
We denote by $R_u\tilde{G}$  the unipotent radical of $\tilde{G}$ so that the  maximal reductive quotient of $\tilde{G}$ is $\tilde{G}/R_u\tilde{G}$.
That is, there is the following exact sequence of group schemes over $\kappa$:
\[1 \longrightarrow R_u\tilde{G} \longrightarrow \tilde{G} \longrightarrow \tilde{G}/R_u\tilde{G} \longrightarrow 1. \]
Furthermore, the following sequence of groups
\[1 \longrightarrow R_u\tilde{G}(\kappa) \longrightarrow \tilde{G}(\kappa) \longrightarrow (\tilde{G}/R_u\tilde{G})(\kappa) \longrightarrow 1. \]
is also exact by Lemma 6.3.3 in \cite{GY}.
Lemma 6.3.3 in \cite{GY} also induces that  $\sharp R_u\tilde{G}(\kappa)$ is $f^m$, where $m$ is the dimension of $R_u\tilde{G}$.
Notice that  the dimension of $R_u\tilde{G}$ can be computed explicitly based on Theorem 4.12 or Remark 4.13,
since the dimension of $\tilde{G}$ is $\frac{n(n-1)}{2}$ with $n=\mathrm{rank}_{A}L$. 
In addition, the order of an orthogonal group  defined over a finite field is well known.
Thus, one can compute $\sharp (\tilde{G}/R_u\tilde{G})(\kappa)$ explicitly based on Theorem 4.12 or Remark 4.13.
Finally, the order of the group $\tilde{G}(\kappa)$ is identified as follows:
\[\sharp\tilde{G}(\kappa)=\sharp R_u\tilde{G}(\kappa)\cdot \sharp (\tilde{G}/R_u\tilde{G})(\kappa).\]

\item As  in Remark 7.4 of \cite{GY}, although we have assumed that $n_i=0$ for $i<0$,
it is easy to check that the formula in the preceding theorem remains true without this assumption.
\end{enumerate}
\end{Rmk}

\section{The mass formula for $Q_n(x_1, \cdots, x_n)=x_1^2 + \cdots + x_n^2$}

  Let us apply the local density formula to obtain the mass formula for the integral quadratic form
 $$Q_n(x_1, \cdots, x_n)=x_1^2 + \cdots + x_n^2.$$
 As we are working globally, we differ from our previous notation at times.
 Let $k$ be a totally real number field of degree $d$ over $\mathbb{Q}$ and $R$ be its ring of integers.
Assume that the ideal $(2)$ is unramified over $R$.
For a place $v$ of $k$, we let $k_v$ be the corresponding completion of $k$.
For a finite place $v$ of $k$, let $\kappa_v$ denote the residue field of the completion of $k$ at $v$,
and $q_v$ be the cardinality of $\kappa_v$.
We consider the quadratic $R$-lattice $(L, Q)$ such that $Q_n(x_1, \cdots, x_n)=x_1^2 + \cdots + x_n^2$.
The paper \cite{GY} of Gan and Yu includes a complete discussion of the Smith-Minkowski-Siegel mass formula.
Applying it to the quadratic lattice $(L, Q)$, we have the following formula:

\begin{Prop}[(\cite{GY}, Theorem 10.20)]
$$\mathrm{Mass}(L, Q_n)=c(L)\cdot \frac{d_k^{n(n-1)/4}}{\prod_{v~finite}\beta_{L_v}}.$$
Here, $c(L)=(\lambda^{-1}\mu)^d  $, $d_k$ is the discriminant of $k$ over $\mathbb{Q}$,
and $\beta_{L_v}$ is the local density associated to the quadratic lattice $L_v$.\\
$\lambda=\prod_i\frac{(2\pi)^{d_i}}{(d_i-1)!}$,  $d_i$'s run over the degrees of $G$.\\
   $\mu=2^n (\mathrm{resp.~} 2^{(n+1)/2})$ if $n$ is even (resp. $n$ is odd).
\end{Prop}

We define $\mathcal{I}_1=\{v : \text{$2|q_v$ and $[\kappa_v : F_2]$ is odd}\}$
and $\mathcal{I}_2=\{v : \text{$2|q_v$ and $[\kappa_v : F_2]$ is even}\}$, where $F_q$ is the finite field with $q$ elements.
Let $\mathcal{I}=\mathcal{I}_1 \cup \mathcal{I}_2$.\\
Based on Theorem 7.3 of \cite{GY} and Theorem 5.2 of this paper, we have the following formula for the local density:
\[
 \left\{
  \begin{array}{l l}
  \beta_{L_v}=1/2 q_v^{-n(n-1)/2}\cdot\sharp\mathrm{O}(n, Q_n) & \quad \text{if $v \nmid (2)$};\\
 \beta_{L_v}=1/2q_v^{1-n(n-1)/2}\cdot\sharp\tilde{G}(\kappa_v) & \quad \text{If $v \mid (2)$},
     \end{array} \right.
\]
 where
\[
  \sharp\tilde{G}(\kappa_v) = \left\{
  \begin{array}{l l}
    4q_v^{n-1}\cdot\sharp\mathrm{SO}(n-1) & \quad \text{if $n\equiv \pm 1$ mod 8, or $n\equiv \pm 3$ mod 8 and $v\in \mathcal{I}_2$};\\
    4q_v^{n-1}\cdot\sharp{}^2\mathrm{SO}(n-1) & \quad \text{if $n\equiv \pm 3$ mod 8 and $v\in \mathcal{I}_1$};\\
    4q_v^{n-1}\cdot\sharp\mathrm{SO}(n-1) & \quad \text{if $n\equiv \pm 2$ mod 8};\\
    4q_v^{2n-3}\cdot\sharp\mathrm{SO}(n-2) & \quad \text{if $n\equiv \pm 0$  mod 8, or $n\equiv  4$ mod 8 and $v\in \mathcal{I}_2$};\\
    4q_v^{2n-3}\cdot\sharp{}^2\mathrm{SO}(n-2) & \quad \text{if $n\equiv  4$ mod 8 and $v\in \mathcal{I}_1$}.\\
    \end{array} \right.
\]
Here, $\mathrm{SO}(n)$ (resp.,${}^2\mathrm{SO}(n)$) denotes the split (resp., nonsplit) connected orthogonal group over $\kappa_v$.

The order of an orthogonal group defined over a finite field is well known. We state it below according to the characteristic of a finite field and the dimension $n$.

If the characteristic of the finite field $\kappa_v$ is greater than 2, then the order of an orthogonal group is as follows:\\
$
   \left\{
  \begin{array}{l l}
   \sharp\mathrm{O}(2m+1, Q_{2m+1})=2q_v^{m^2}\prod_{i=1}^{m}(q_v^{2i}-1)\\
   \sharp\mathrm{O}(2m, Q_{2m})=2q_v^{m(m-1)}(q_v^m-1)\prod_{i=1}^{m-1}(q_v^{2i}-1) &  \text{if $-1$ is a square in $F_q$}\\
   \sharp\mathrm{O}(2m, Q_{2m})=2q_v^{m(m-1)}(q_v^m-(-1)^m)\prod_{i=1}^{m-1}(q_v^{2i}-1) &  \text{if $-1$ is not a square in $F_q$}.
    \end{array} \right.
$ \\

If the characteristic of the finite field $\kappa_v$ is 2, then the order of an orthogonal group is as follows:\\
$
   \left\{
  \begin{array}{l l}
\sharp\mathrm{SO}(2m+1)=q_v^{m^2}\prod_{i=1}^{m}(q_v^{2i}-1)\\
\sharp\mathrm{SO}(2m)=q_v^{m(m-1)}(q_v^m-1)\prod_{i=1}^{m-1}(q_v^{2i}-1)\\
\sharp{}^2\mathrm{SO}(2m)=q_v^{m(m-1)}(q_v^m+1)\prod_{i=1}^{m-1}(q_v^{2i}-1).
    \end{array} \right.
$\\

By combining these with Proposition 6.1, we have the following theorem:
\begin{Thm}
\begin{itemize}
\item[(1)]
 When $n=2m+1$,
\begin{equation}
\mathrm{Mass}(L, Q_n)= ((\prod^{m}_{i=1}\frac{(2i-1)!}{(2\pi)^{2i}})\cdot 2^{m+1})^d   \cdot d_k^{n(n-1)/4} \cdot D(L) \cdot \prod^{m}_{i=1}\zeta_k(2i),
\end{equation}
\[
D(L) = \left\{
  \begin{array}{l l}
  \prod^{~}_{v\in\mathcal{I}}\frac{q_v^m+1}{2q_v^{m+1}}    & \quad  \text{if $n\equiv \pm 1$ mod 8};\\
  \prod_{v\in\mathcal{I}_1}^{ }\frac{q_v^m-1}{2q_v^{m+1}}\cdot \prod_{v\in\mathcal{I}_2}^{ }\frac{q_v^m+1}{2q_v^{m+1}}   &   \quad  \text{if $n\equiv \pm 3$ mod 8}.\\
    \end{array} \right.
\]
\item[(2)]
 When $n=2m$,
\begin{equation}
\mathrm{Mass}(L, Q_n)=
((\prod^{m-1}_{i=1}\frac{(2i-1)!}{(2\pi)^{2i}})\cdot \frac{(m-1)!}{(2\pi)^m} \cdot 2^{2m})^d   \cdot d_k^{n(n-1)/4} \cdot D(L) \cdot L_k(m, \chi)  \cdot \prod^{m-1}_{i=1}\zeta_k(2i),
\end{equation}
\[
D(L)= \left\{
  \begin{array}{l l}
  \prod^{~}_{v\in\mathcal{I}}\frac{(q_v^{m-1}+1)(q_v^m-1)}{2q_v^{2m}}    & \quad  \text{if $n\equiv 0$ mod 8};\\
  \prod_{v\in\mathcal{I}_1}^{ }\frac{(q_v^{m-1}-1)(q_v^m-1)}{2q_v^{2m}}\cdot \prod_{v\in\mathcal{I}_2}^{ }\frac{(q_v^{m-1}+1)(q_v^m-1)}{2q_v^{2m}}  &  \quad
  \text{if $n\equiv 4$ mod 8};\\
  \prod_{v\in\mathcal{I} }^{ }                            
  \frac{1}{2q_v}
    & \quad  \text{if $n\equiv \pm 2 $ mod 8}.\\
    \end{array} \right.
\]
Here $\chi$ is the Galois character of $k(\sqrt{(-1)^m})$ over $k$.
\end{itemize}
\end{Thm}




Since the mass formula represents a rational number,
we can rewrite the above formulas using the functional equations of the Dedekind zeta function and the Hecke $L$-series.
For the functional equation of the Dedekind zeta function, we refer to \cite{N}.

\begin{Prop}[(\cite{N}, Corollary VII.5.10)]
Let $\zeta_k(s)$ be the Dedekind zeta function of the totally real number field $k$.
The completed zeta function
$$Z_k(s)=d_k^{s/2}\cdot(\pi^{-s/2}\Gamma(s/2))^d\cdot\zeta_k(s)$$
satisfies the functional equation
$$Z_k(s)=Z_k(1-s).$$
\end{Prop}

If $s=2i$ for an integer $i$, the above proposition gives the following equation:
\begin{equation}
\zeta_k(1-2i)=d_k^{2i-1/2}\cdot(\pi^{-2i+1/2}\frac{\Gamma(i)}{\Gamma(\frac{1}{2}-i)})^d\cdot\zeta_k(2i),
\end{equation}
where $\Gamma(i)=(i-1)!$ and $\Gamma(\frac{1}{2}-i)=\frac{(-4)^i\cdot i!}{(2i)!}\sqrt{\pi}$.
Therefore, $$\frac{\Gamma(i)}{\Gamma(\frac{1}{2}-i)}=2\cdot(2i-1)!\cdot(-4)^{-i}\cdot\pi^{-1/2}.$$\\

We first assume that $n=2m+1$. Then  Equation  (6.6) induces the following equation:
$$\prod_{i=1}^{m}\zeta_k(1-2i)=(d_k)^{m^2+m/2}\cdot2^{md}\cdot(-1)^{m(m+1)d/2}\cdot\prod_{i=1}^{m}((\frac{(2i-1)!}{(2\pi)^{2i}})^d\cdot\zeta_k(2i)).$$
If we apply the above to  Equation (6.3), then we have the following Mass formula:
\begin{equation}
\mathrm{Mass}(L, Q_n)=(-1)^{m(m+1)d/2}\cdot 2^{d}  \cdot D(L) \cdot \prod^{m}_{i=1}\zeta_k(1-2i).\\
\end{equation}\\

We next assume that $n=2m$. Then  Equation  (6.6) induces the following equation:
$$\prod_{i=1}^{m-1}\zeta_k(1-2i)=(d_k)^{m^2-m/2}\cdot2^{md-d}\cdot(-1)^{m(m-1)d/2}\cdot\prod_{i=1}^{m-1}((\frac{(2i-1)!}{(2\pi)^{2i}})^d\cdot\zeta_k(2i)).$$
If we apply the above to  Equation (6.4), then we have the following Mass formula:
\begin{equation}
\mathrm{Mass}(L, Q_n)=
 d_k^{m-1/2}\cdot \pi^{-md}\cdot (-1)^{m(m-1)d/2}\cdot ((m-1)!)^d  \cdot  2^{d}\cdot D(L)\cdot  L_k(m,\chi)\cdot \prod^{m-1}_{i=1}\zeta_k(1-2i).
\end{equation}\\

Let us apply the functional equation of the Hecke $L$-series to  Equation (6.8).
If $m$ is even, equivalently $n\equiv 0$ or $4$ mod 8, then the character $\chi$ is trivial.
Let $m=2m^{\prime}$. If we put $i=m^{\prime}$ in Equation (6.6), we have the following:
$$\zeta_k(1-m)=d_k^{m-1/2}\cdot(\pi^{-m+1/2}\frac{\Gamma(m^{\prime})}{\Gamma(1/2-m^{\prime})})^d\cdot\zeta_k(m).$$
Equivalently,
\begin{equation}
\zeta_k(1-m)=d_k^{m-1/2}\cdot 2^d\cdot (\frac{(m-1)!}{(-1)^{m^{\prime}}\cdot(2\pi)^m}   )^d\cdot\zeta_k(m).
\end{equation}
We apply Equation (6.9) to Equation (6.8). Then we have the following Mass formula:
\begin{equation}
\mathrm{Mass}(L, Q_n)= 2^{md}\cdot D(L)\cdot \zeta_k(1-m) \cdot \prod^{m-1}_{i=1}\zeta_k(1-2i).
\end{equation}\\

Assume that $m$ $(=2m^{\prime}+1)$ is odd, equivalently $n\equiv \pm 2 $ mod 8.
In this case, the Galois character $\chi$ is non-trivial.
We state the   functional equation for the Hecke $L$-series.

\begin{Prop}[(\cite{N}, Corollary VII.8.6)]
The completed Hecke $L$-series
$$\Lambda_k(s, \chi)=(d_k\cdot N_{k/\mathbb{Q}}\mathfrak{f}(\chi))^{s/2}\cdot (\pi^{\frac{-(s+1)}{2}}\cdot\Gamma(\frac{s+1}{2}))^d\cdot L_k(s, \chi)$$
satisfies the functional equation
$$\Lambda_k(s, \chi)=\epsilon(\chi)\cdot\Lambda_k(1-s, \chi).$$
Here, $\mathfrak{f}(\chi)$ is the conductor of the Hecke character $\chi$
and $|\epsilon(\chi)|=1$.
\end{Prop}

If $s=m$ in the above proposition, we have the following equation:
\begin{equation}
L_k(m, \chi)=\epsilon(\chi)\cdot L_k(1-m, \chi)\cdot (d_k\cdot N_{k/\mathbb{Q}}\mathfrak{f}(\chi))^{\frac{1-n}{2}}\cdot \pi^{md}\cdot\frac{(-4)^{m^{\prime}d}}{((m-1)!)^d}.
\end{equation}
Let us apply Equation (6.12) to Equation (6.8). Then we have the following Mass formula:
\begin{equation}
\mathrm{Mass}(L, Q_n)=2^{md}\cdot D(L)\cdot \epsilon(\chi)\cdot (N_{k/\mathbb{Q}}\mathfrak{f}(\chi))^{(1-n)/2}\cdot L_k(1-m,\chi)\cdot \prod^{m-1}_{i=1}\zeta_k(1-2i).
\end{equation}\\

By combining Equation (6.7), (6.10) and (6.13), we finally have the following theorem:

\begin{Thm}
\begin{itemize}
\item[(1)] When $n=2m+1$,
$$\mathrm{Mass}(L, Q_n)=(-1)^{m(m+1)d/2}\cdot  d_k^m \cdot 2^{d}  \cdot D(L) \cdot \prod^{m}_{i=1}\zeta_k(1-2i).$$

\item[(2)] When $n=2m$,
$$\mathrm{Mass}(L, Q_n)=2^{md}\cdot D(L)\cdot \epsilon(\chi)\cdot (N_{k/\mathbb{Q}}\mathfrak{f}(\chi))^{(1-n)/2}\cdot L_k(1-m,\chi)\cdot \prod^{m-1}_{i=1}\zeta_k(1-2i).$$
Here, $|\epsilon(\chi)|=1$ and $D(L)$ is defined in Theorem 6.2.
\end{itemize}
\end{Thm}

\section{Appendix: The proof of  Lemma 4.2}

\begin{proof}
Recall that  $\tilde{M}$ is the special fiber of $\underline{M}^{\ast}$.
As similar to the construction of $\varphi_i$ explained at the beginning of Section 4.1,
the morphism $\varphi_i$ is extended to the morphism
 $$\tilde{\varphi_i} : \tilde{M} \longrightarrow  \mathrm{Aut}_{\kappa}(\bar{V_i}) $$
 such that $\tilde{\varphi_i}|_{\tilde{G}}=\varphi_i$. Here, $\bar{V_i}=B_i/Z_i$.
We define $$\tilde{\varphi}=\prod_i \tilde{\varphi_i} : \tilde{M} \longrightarrow  \prod_i\mathrm{Aut}_{\kappa}(\bar{V_i}).$$
Then $\tilde{\varphi}|_{\tilde{G}}=\varphi$.


Before describing the equations defining $\mathrm{Ker~}\tilde{\varphi} $, we state notations here.
We use $a_{i-1}, b_{i-1}, c_{i-1},$ $d_{i-1}, e_{i-1}, f_{i-1}, g_{i-1}, h_{i-1}, i_{i-1}$
(resp. $a_{i-1}^{\prime}, b_{i-1}^{\prime}, c_{i-1}^{\prime},
d_{i-1}^{\prime}, e_{i-1}^{\prime},
 f_{i-1}^{\prime}, g_{i-1}^{\prime}, h_{i-1}^{\prime}, i_{i-1}^{\prime}$)
to denote a block in $m_{i-1, i}$  (resp. $m_{i, i-1}$).\\
Recall that we have represented the given quadratic form $q$ by a symmetric matrix $\begin{pmatrix} 2^{i}\cdot \delta_i\end{pmatrix}$ with $2^{i}\cdot \delta_i$
     for the $(i,i)$-block and $0$ for remaining blocks.
 Assume that $L_i$ is \textit{of type I}. Let $\delta_i=\begin{pmatrix}  \delta_i^{\prime}&0\\0&\delta_i^{\prime\prime}  \end{pmatrix}$,
 where $\delta_{i}^{\prime}$ is an $(n_{i}-1) \times (n_i-1)$-matrix (resp. an $(n_{i}-2) \times (n_i-2)$-matrix)
 if $L_i$ is \textit{of type} $\textit{I}^o$ (resp. \textit{of type} $\textit{I}^e$).
In particular, if $L_i$ is \textit{of type} $\textit{I}^e$,
 $\delta_i^{\prime\prime}=\begin{pmatrix} 1&1\\1&2\gamma_i  \end{pmatrix}$ by Theorem 2.4. 
We denote  $\gamma_i$ mod $2$ by  $\widetilde{\gamma_i}(\in \kappa)$.
In addition, we denote  the solution of the equation $x^2-\widetilde{\gamma_i}=0$ by $\sqrt{\widetilde{\gamma_i}} (\in \kappa)$.\\

We now describe the equations defining $\mathrm{Ker~}\tilde{\varphi} $. There are the following 9 cases according to each type of $L_{i-1}, L_i, L_{i+1}$.
\begin{enumerate}
\item Assume that $L_i$ is \textit{free}.
\begin{itemize}
\item[a)] If $L_i$ is  \textit{of type II},  set $m_{i,i}=\mathrm{id}$.
\item[b)] If $L_i$ is  \textit{of type} $\textit{I}^o$, set $s_i=\mathrm{id}$.
\item[c)] If $L_i$ is  \textit{of type} $\textit{I}^e_1$, set $s_i=\mathrm{id}$ and $v_i=0$.
\item[d)] If $L_i$ is  \textit{of type} $\textit{I}^e_2$, set $s_i=\mathrm{id}$.
\end{itemize}

\item Assume that $L_{i-1}$ is \textit{of type} $\textit{I}^o$ and $L_{i+1}$ is \textit{of type II}.
\begin{itemize}
\item[a)] If $L_i$ is  \textit{of type II}, the matrix form of $2m_{i-1, i}$ is $\begin{pmatrix} 2a_{i-1}\\2b_{i-1}  \end{pmatrix}$,
where $a_{i-1}$ is an $(n_{i-1}-1) \times (n_i)$-matrix, etc.
Set $m_{i,i}=\mathrm{id}$ and $b_{i-1}=0$.
\item[b)] If $L_i$ is  \textit{of type} $\textit{I}^o$, the matrix form of $2m_{i-1, i}$ is $\begin{pmatrix} 2a_{i-1}&2c_{i-1}\\ 2b_{i-1}&4f_{i-1}  \end{pmatrix}$,
where $a_{i-1}$ is an $(n_{i-1}-1) \times (n_i-1)$-matrix, etc.
Set $s_i=\mathrm{id}$ and $b_{i-1}=0$.
\item[c)] If $L_i$ is  \textit{of type} $\textit{I}^e$, the matrix form of $2m_{i-1, i}$ is $\begin{pmatrix} 2a_{i-1}&2c_{i-1} & 2e_{i-1}\\ 2b_{i-1}&2d_{i-1}&4f_{i-1}  \end{pmatrix}$,
where $a_{i-1}$ is an $(n_{i-1}-1) \times (n_i-2)$-matrix, etc.
If $L_i$ is  \textit{of type} $\textit{I}^e$, set $s_i=\mathrm{id}$ and $b_{i-1}=\sqrt{\widetilde{\gamma_i}}\cdot v_i$.
\end{itemize}

\item Assume that $L_{i-1}$ is \textit{of type} $\textit{I}^e$ and $L_{i+1}$ is \textit{of type II}.
\begin{itemize}
\item[a)] If $L_i$ is  \textit{of type II}, the matrix form of $2m_{i-1, i}$ is $\begin{pmatrix} 2a_{i-1}\\2b_{i-1} \\2c_{i-1} \end{pmatrix}$,
where $a_{i-1}$ is an $(n_{i-1}-2) \times (n_i)$-matrix, etc.
Set $m_{i,i}=\mathrm{id}$ and $b_{i-1}=0$.
\item[b)] If $L_i$ is  \textit{of type} $\textit{I}^o$, the matrix form of $2m_{i-1 i}$ is $\begin{pmatrix} 2a_{i-1}&2d_{i-1}\\ 2b_{i-1}&4f_{i-1}\\ 2c_{i-1}&2e_{i-1}  \end{pmatrix}$,
where $a_{i-1}$ is an $(n_{i-1}-2) \times (n_i-1)$-matrix, etc.
Set $s_i=\mathrm{id}$ and $b_{i-1}=0$.
\item[c)] If $L_i$ is  \textit{of type} $\textit{I}^e$, the matrix form of $2m_{i-1, i}$ is $\begin{pmatrix} 2a_{i-1}&2i_{i-1}& 2h_{i-1}\\2b_{i-1}&2d_{i-1}&4f_{i-1} \\2c_{i-1}&2g_{i-1}&2e_{i-1} \end{pmatrix}$,
where $a_{i-1}$ is an $(n_{i-1}-2) \times (n_i-2)$-matrix, etc.
If $L_i$ is  \textit{of type} $\textit{I}^e$, set $s_i=\mathrm{id}$ and $b_{i-1}=\sqrt{\widetilde{\gamma_i}}\cdot v_i$.
\end{itemize}

\item Assume that $L_{i-1}$ is \textit{of type II} and $L_{i+1}$ is \textit{of type} $\textit{I}^o$.
\begin{itemize}
\item[a)] If $L_i$ is  \textit{of type II}, the matrix form of $m_{i+1, i}$ is $\begin{pmatrix} a_{i}^{\prime}\\b_{i}^{\prime}  \end{pmatrix}$,
where $a_{i}^{\prime}$ is an $(n_{i+1}-1) \times (n_i)$-matrix, etc.
Set $m_{i,i}=\mathrm{id}$ and $b_{i}^{\prime}=0$.
\item[b)] If $L_i$ is  \textit{of type} $\textit{I}^o$, the matrix form of $m_{i+1, i}$ is $\begin{pmatrix} a_{i}^{\prime}&c_{i}^{\prime}\\ b_{i}^{\prime}&2f_{i}^{\prime}  \end{pmatrix}$,
where $a_{i}^{\prime}$ is an $(n_{i+1}-1) \times (n_i-1)$-matrix, etc.
Set $s_i=\mathrm{id}$ and $b_{i}^{\prime}=0$.
\item[c)] If $L_i$ is  \textit{of type} $\textit{I}^e$, the matrix form of $m_{i+1, i}$ is $\begin{pmatrix} a_{i}^{\prime}&c_{i}^{\prime} & e_{i}^{\prime}\\ b_{i}^{\prime}&d_{i}^{\prime}&2f_{i}^{\prime}  \end{pmatrix}$,
where $a_{i}^{\prime}$ is an $(n_{i+1}-1) \times (n_i-2)$-matrix, etc.
If $L_i$ is  \textit{of type} $\textit{I}^e$, set $s_i=\mathrm{id}$ and $b_{i}^{\prime}=\sqrt{\widetilde{\gamma_i}}\cdot v_i$.
\end{itemize}

\item Assume that $L_{i-1}$ is \textit{of type II} and $L_{i+1}$ is \textit{of type} $\textit{I}^e$.
\begin{itemize}
\item[a)] If $L_i$ is  \textit{of type II}, the matrix form of $m_{i+1, i}$ is $\begin{pmatrix} a_{i}^{\prime}\\b_{i}^{\prime} \\c_{i}^{\prime} \end{pmatrix}$,
where $a_{i}^{\prime}$ is an $(n_{i+1}-2) \times (n_i)$-matrix, etc.
Set $m_{ii}=\mathrm{id}$ and $b_{i}^{\prime}=0$.
\item[b)] If $L_i$ is  \textit{of type} $\textit{I}^o$, the matrix form of $m_{i+1, i}$ is $\begin{pmatrix} a_{i}^{\prime}&d_{i}^{\prime}\\ b_{i}^{\prime}&2f_{i}^{\prime}\\ c_{i}^{\prime}&e_{i}^{\prime}  \end{pmatrix}$,
where $a_{i}^{\prime}$ is an $(n_{i+1}-2) \times (n_i-1)$-matrix, etc.
Set $s_i=\mathrm{id}$ and $b_{i}^{\prime}=0$.
\item[c)] If $L_i$ is  \textit{of type} $\textit{I}^e$, the matrix form of $m_{i+1, i}$ is $\begin{pmatrix} a_{i}^{\prime}&i_{i}^{\prime}& h_{i}^{\prime}\\b_{i}^{\prime}&d_{i}^{\prime}&2f_{i}^{\prime} \\c_{i}^{\prime}&g_{i}^{\prime}&e_{i}^{\prime} \end{pmatrix}$,
where $a_{i}^{\prime}$ is an $(n_{i+1}-2) \times (n_i-2)$-matrix, etc.
If $L_i$ is  \textit{of type} $\textit{I}^e$, set $s_i=\mathrm{id}$ and $b_{i}^{\prime}=\sqrt{\widetilde{\gamma_i}}\cdot v_i$.
\end{itemize}

\item Assume that $L_{i-1}$ is \textit{of type} $\textit{I}^o$ and $L_{i+1}$ is \textit{of type} $\textit{I}^o$.
\begin{itemize}
\item[a)] If $L_i$ is  \textit{of type II}, the matrix form of $2m_{i-1, i}$ is $\begin{pmatrix} 2a_{i-1}\\2b_{i-1}  \end{pmatrix}$,
where $a_{i-1}$ is an $(n_{i-1}-1) \times (n_i)$-matrix, and
the matrix form of $m_{i+1, i}$ is $\begin{pmatrix} a_{i}^{\prime}\\b_{i}^{\prime}  \end{pmatrix}$,
where $a_{i}^{\prime}$ is an $(n_{i+1}-1) \times (n_i)$-matrix, etc.
Set $m_{i,i}=\mathrm{id}$ and $b_{i-1}=b_{i}^{\prime}$.
\item[b)] If $L_i$ is  \textit{of type} $\textit{I}^o$, the matrix form of $2m_{i-1, i}$ is $\begin{pmatrix} 2a_{i-1}&2c_{i-1}\\ 2b_{i-1}&4f_{i-1}  \end{pmatrix}$,
where $a_{i-1}$ is an $(n_{i-1}-1) \times (n_i-1)$-matrix, and
the matrix form of $m_{i+1, i}$ is $\begin{pmatrix} a_{i}^{\prime}&c_{i}^{\prime}\\ b_{i}^{\prime}&2f_{i}^{\prime}  \end{pmatrix}$,
where $a_{i}^{\prime}$ is an $(n_{i+1}-1) \times (n_i-1)$-matrix, etc.
Set $s_i=\mathrm{id}$ and $b_{i-1}=b_{i}^{\prime}$.
\item[c)] If $L_i$ is  \textit{of type} $\textit{I}^e$, the matrix form of $2m_{i-1, i}$ is $\begin{pmatrix} 2a_{i-1}&2c_{i-1} & 2e_{i-1}\\ 2b_{i-1}&2d_{i-1}&4f_{i-1}  \end{pmatrix}$,
where $a_{i-1}$ is an $(n_{i-1}-1) \times (n_i-2)$-matrix, and
the matrix form of $m_{i+1, i}$ is $\begin{pmatrix} a_{i}^{\prime}&c_{i}^{\prime} & e_{i}^{\prime}\\ b_{i}^{\prime}&d_{i}^{\prime}&2f_{i}^{\prime}  \end{pmatrix}$,
where $a_{i}^{\prime}$ is an $(n_{i+1}-1) \times (n_i-2)$-matrix, etc.
If $L_i$ is  \textit{of type} $\textit{I}^e$, set $s_i=\mathrm{id}$ and $b_{i-1}+\sqrt{\widetilde{\gamma_i}}\cdot v_i+b_{i}^{\prime}=0$.
\end{itemize}

\item Assume that $L_{i-1}$ is \textit{of type} $\textit{I}^e$ and $L_{i+1}$ is \textit{of type} $\textit{I}^o$.
\begin{itemize}
\item[a)] If $L_i$ is  \textit{of type II}, the matrix form of $2m_{i-1, i}$ is $\begin{pmatrix} 2a_{i-1}\\2b_{i-1} \\2c_{i-1} \end{pmatrix}$,
where $a_{i-1}$ is an $(n_{i-1}-2) \times (n_i)$-matrix,  and
the matrix form of $m_{i+1, i}$ is $\begin{pmatrix} a_{i}^{\prime}\\b_{i}^{\prime}  \end{pmatrix}$,
where $a_{i}^{\prime}$ is an $(n_{i+1}-1) \times (n_i)$-matrix, etc.
Set $m_{ii}=\mathrm{id}$ and $b_{i-1}=b_{i}^{\prime}$.
\item[b)] If $L_i$ is  \textit{of type} $\textit{I}^o$, the matrix form of $2m_{i-1, i}$ is $\begin{pmatrix} 2a_{i-1}&2d_{i-1}\\ 2b_{i-1}&4f_{i-1}\\ 2c_{i-1}&2e_{i-1}  \end{pmatrix}$,
where $a_{i-1}$ is an $(n_{i-1}-2) \times (n_i-1)$-matrix,  and
the matrix form of $m_{i+1, i}$ is $\begin{pmatrix} a_{i}^{\prime}&c_{i}^{\prime}\\ b_{i}^{\prime}&2f_{i}^{\prime}  \end{pmatrix}$,
where $a_{i}^{\prime}$ is an $(n_{i+1}-1) \times (n_i-1)$-matrix, etc.
Set $s_i=\mathrm{id}$ and $b_{i-1}=b_{i}^{\prime}$.
\item[c)] If $L_i$ is  \textit{of type} $\textit{I}^e$, the matrix form of $2m_{i-1, i}$ is $\begin{pmatrix} 2a_{i-1}&2i_{i-1}& 2h_{i-1}\\2b_{i-1}&2d_{i-1}&4f_{i-1} \\2c_{i-1}&2g_{i-1}&2e_{i-1} \end{pmatrix}$,
where $a_{i-1}$ is an $(n_{i-1}-2) \times (n_i-2)$-matrix,  and
the matrix form of $m_{i+1, i}$ is $\begin{pmatrix} a_{i}^{\prime}&c_{i}^{\prime} & e_{i}^{\prime}\\ b_{i}^{\prime}&d_{i}^{\prime}&2f_{i}^{\prime}  \end{pmatrix}$,
where $a_{i}^{\prime}$ is an $(n_{i+1}-1) \times (n_i-2)$-matrix, etc.
If $L_i$ is  \textit{of type} $\textit{I}^e$, set $s_i=\mathrm{id}$ and $b_{i-1}+\sqrt{\widetilde{\gamma_i}}\cdot v_i+b_{i}^{\prime}=0$.
\end{itemize}

\item Assume that $L_{i-1}$ is \textit{of type} $\textit{I}^o$ and $L_{i+1}$ is \textit{of type} $\textit{I}^e$.
\begin{itemize}
\item[a)] If $L_i$ is  \textit{of type II},
the matrix form of $2m_{i-1, i}$ is $\begin{pmatrix} 2a_{i-1}\\2b_{i-1}  \end{pmatrix}$,
where $a_{i}^{\prime}$ is an $(n_{i-1}-1) \times (n_i)$-matrix, and
the matrix form of $m_{i+1, i}$ is $\begin{pmatrix} a_{i}^{\prime}\\b_{i}^{\prime} \\c_{i}^{\prime} \end{pmatrix}$,
where $a_{i}^{\prime}$ is an $(n_{i+1}-2) \times (n_i)$-matrix,
Set $m_{i,i}=\mathrm{id}$ and $b_{i-1}=b_{i}^{\prime}$.
\item[b)] If $L_i$ is  \textit{of type} $\textit{I}^o$,
the matrix form of $2m_{i-1, i}$ is $\begin{pmatrix} 2a_{i-1}&2c_{i-1}\\ 2b_{i-1}&4f_{i-1}  \end{pmatrix}$,
where $a_{i-1}$ is an $(n_{i-1}-1) \times (n_i-1)$-matrix, and
the matrix form of $m_{i+1, i}$ is $\begin{pmatrix} a_{i}^{\prime}&d_{i}^{\prime}\\ b_{i}^{\prime}&2f_{i}^{\prime}\\ c_{i}^{\prime}&e_{i}^{\prime}  \end{pmatrix}$,
where $a_{i}^{\prime}$ is an $(n_{i+1}-2) \times (n_i-1)$-matrix.
Set $s_i=\mathrm{id}$ and $b_{i-1}=b_{i}^{\prime}$.
\item[c)] If $L_i$ is  \textit{of type} $\textit{I}^e$,
the matrix form of $2m_{i-1, i}$ is $\begin{pmatrix} 2a_{i-1}&2c_{i-1} & 2e_{i-1}\\ 2b_{i-1}&2d_{i-1}&4f_{i-1}  \end{pmatrix}$,
where $a_{i-1}$ is an $(n_{i-1}-1) \times (n_i-2)$-matrix, and
the matrix form of $m_{i+1, i}$ is $\begin{pmatrix} a_{i}^{\prime}&i_{i}^{\prime}& h_{i}^{\prime}\\b_{i}^{\prime}&d_{i}^{\prime}&2f_{i}^{\prime} \\c_{i}^{\prime}&g_{i}^{\prime}&e_{i}^{\prime} \end{pmatrix}$,
where $a_{i}^{\prime}$ is an $(n_{i+1}-2) \times (n_i-2)$-matrix.
If $L_i$ is  \textit{of type} $\textit{I}^e$, set $s_i=\mathrm{id}$ and $b_{i-1}+\sqrt{\widetilde{\gamma_i}}\cdot v_i+b_{i}^{\prime}=0$.
\end{itemize}

\item Assume that $L_{i-1}$ is \textit{of type} $\textit{I}^e$ and $L_{i+1}$ is \textit{of type} $\textit{I}^e$.
\begin{itemize}
\item[a)] If $L_i$ is  \textit{of type II}, the matrix form of $2m_{i-1, i}$ is $\begin{pmatrix} 2a_{i-1}\\2b_{i-1} \\2c_{i-1} \end{pmatrix}$,
where $a_{i-1}$ is an $(n_{i-1}-2) \times (n_i)$-matrix,  and
the matrix form of $m_{i+1, i}$ is $\begin{pmatrix} a_{i}^{\prime}\\b_{i}^{\prime} \\c_{i}^{\prime} \end{pmatrix}$,
where $a_{i}^{\prime}$ is an $(n_{i+1}-2) \times (n_i)$-matrix, etc.
Set $m_{i,i}=\mathrm{id}$ and $b_{i-1}=b_{i}^{\prime}$.
\item[b)] If $L_i$ is  \textit{of type} $\textit{I}^o$, the matrix form of $2m_{i-1, i}$ is $\begin{pmatrix} 2a_{i-1}&2d_{i-1}\\ 2b_{i-1}&4f_{i-1}\\ 2c_{i-1}&2e_{i-1}  \end{pmatrix}$,
where $a_{i-1}$ is an $(n_{i-1}-2) \times (n_i-1)$-matrix,  and
the matrix form of $m_{i+1, i}$ is $\begin{pmatrix} a_{i}^{\prime}&d_{i}^{\prime}\\ b_{i}^{\prime}&2f_{i}^{\prime}\\ c_{i}^{\prime}&e_{i}^{\prime}  \end{pmatrix}$,
where $a_{i}^{\prime}$ is an $(n_{i+1}-2) \times (n_i-1)$-matrix, etc.
Set $s_i=\mathrm{id}$ and $b_{i-1}=b_{i}^{\prime}$.
\item[c)] If $L_i$ is  \textit{of type} $\textit{I}^e$, the matrix form of $2m_{i-1, i}$ is $\begin{pmatrix} 2a_{i-1}&2i_{i-1}& 2h_{i-1}\\2b_{i-1}&2d_{i-1}&4f_{i-1} \\2c_{i-1}&2g_{i-1}&2e_{i-1} \end{pmatrix}$,
where $a_{i-1}$ is an $(n_{i-1}-2) \times (n_i-2)$-matrix,  and
the matrix form of $m_{i+1, i}$ is $\begin{pmatrix} a_{i}^{\prime}&i_{i}^{\prime}& h_{i}^{\prime}\\b_{i}^{\prime}&d_{i}^{\prime}&2f_{i}^{\prime} \\c_{i}^{\prime}&g_{i}^{\prime}&e_{i}^{\prime} \end{pmatrix}$,
where $a_{i}^{\prime}$ is an $(n_{i+1}-2) \times (n_i-2)$-matrix, etc.
If $L_i$ is  \textit{of type} $\textit{I}^e$, set $s_i=\mathrm{id}$ and $b_{i-1}+\sqrt{\widetilde{\gamma_i}}\cdot v_i+b_{i}^{\prime}=0$.
\end{itemize}
\end{enumerate}

 Before investigating the equations defining $\mathrm{Ker~}\varphi$, let us introduce the notation $\delta_i^{\prime}(b_i)$ in this paragraph.
 Recall that we say $\delta_i=\begin{pmatrix}  \delta_i^{\prime}&0\\0&\delta_i^{\prime\prime}  \end{pmatrix}$ at the beginning of the Appendix.
 Then the symmetric matrix $\delta_i^{\prime}$ defines a quadratic form which is 0 modulo 2. We define
 $$\delta_i^{\prime}(b_i):=\frac{1}{2}b_i\cdot\delta_i^{\prime}\cdot {}^t b_i$$
   as matrix multiplication, where $b_i$ is a $1 \times (n_{i}-1) $-low vector (resp. $1 \times (n_{i}-2)$-low vector)
   if $L_i$ is \textit{of type} $\textit{I}^o$ (resp. \textit{of type} $\textit{I}^e$).
 If $L_i$ is \textit{of type II}, we define $\delta_i^{\prime}(b_i)=\frac{1}{2}b_i\cdot\delta_i\cdot {}^tb_i$,
 where $b_i$ is a $(1 \times n_{i})$-low vector.\\

We are ready to state the equations defining $\mathrm{Ker~}\varphi$. They are obtained by the matrix equation ${}^tmqm=q$,
where $m$ is an element of $\mathrm{Ker~}\tilde{\varphi}(R)$ for a $\kappa$-algebra $R$.\\
By observing the diagonal $(i,i)$-blocks of  ${}^tmqm=q$, we have the following matrix equation:
\begin{equation}
{}^tm_{i,i}\delta_im_{i,i}+2({}^tm_{i-1, i}\delta_{i-1}m_{i-1, i}+{}^tm_{i+1, i}\delta_{i+1}m_{i+1, i})+
4({}^tm_{i-2, i}\delta_{i-2}m_{i-2, i}+
 {}^tm_{i+2, i}\delta_{i+2}m_{i+2, i})=(\delta_i),
\end{equation}
where $0\leq i < N.$\\
By observing the $(i,i+1)$-blocks of ${}^tmqm=q$, we have the following matrix equation:
\begin{equation}
{}^tm_{i,i}\delta_im_{i,i+1}+{}^tm_{i+1,i}\delta_{i+1}m_{i+1,i+1}+2({}^tm_{i-1,i}\delta_{i-1}m_{i-1,i+1}+{}^tm_{i+2,i}\delta_{i+2}m_{i+2,i+1})=0,
\end{equation}
where $0\leq i < N-1.$\\
By observing the  $(i,j)$-blocks of ${}^tmqm=q$, where $i+2 \leq j$, we have the following matrix equation:
\begin{equation}
\sum_{i\leq k \leq j} {}^tm_{k,i}\delta_km_{k,j}=0,
\end{equation}
where $0\leq i, j < N.$\\

 We first state the equations $\mathcal{F}_i$ and $\mathcal{E}_i$. These equations determine the connected components of $\mathrm{Ker~}\varphi$.
 Assume that $L_i$ is \textit{of type I}. By computing the $(2 \times 2)$-block
 (if $L_i$ is \textit{of type} $\textit{I}^o$) or the $(3 \times 3)$-block (if $L_i$ is \textit{of type} $\textit{I}^e$) of Equation $(7.1)$,  we have the equation $\mathcal{F}_i$:
 \begin{equation}
 \mathcal{F}_i : z_i+z_i^2+\delta_{i-1}^{\prime}(b_{i-1}^{\prime})+\delta_{i+1}^{\prime}(b_i)+
 \widetilde{\gamma_{i-1}}\cdot e_{i-1}^2+\widetilde{\gamma_{i+1}}\cdot d_i^2+x_{i-2}^2+x_{i}^2=0.
 \end{equation}

 Here,
 \begin{itemize}
 \item $b_{i-1}^{\prime}$ and $b_i$ are the blocks in $m_{i, i-1}$ and $m_{i, i+1}$ respectively, as defined above.
 \item $e_{i-1}$ is the $(3\times 2)$-block (if $L_{i}$ is \textit{of type} $\textit{I}^o$) or the $(3 \times 3)$-block (if $L_{i}$ is \textit{of type}
     $\textit{I}^e$) of $m_{i-1, i}$.
 \item $d_i$ is the $(2\times 2)$-block of $m_{i, i+1}$.\\ 
\item $x_{i-2}$  is the $(2\times 2)$-block (resp. the $(2\times 3)$-block) of $m_{i-2, i}$  when $L_{i}$ is \textit{of type} $\textit{I}^o$ (resp. \textit{of type} $\textit{I}^e$)
 and $L_{i-2}$  is \textit{of type I}.
\item $x_{i}$   is the $(2\times 2)$-block or the $(2\times 3)$-block of $m_{i, i+2}$
when $L_{i+2}$ is \textit{of type} $\textit{I}^o$ or \textit{of type} $\textit{I}^e$, respectively.\\
  If $L_{i-2}$ (resp. $L_{i+2}$) is not \textit{of type I}, we remove $x_{i-2}^2$ (resp. $x_i^2$) in the equation $\mathcal{F}_i$.
\end{itemize}
 If $L_i$ is \textit{of type} $\textit{I}^e$, the $(2 \times 2)$-block of  Equation $(7.1)$ induces the equation $\mathcal{E}_i$:
 \begin{equation}
 \mathcal{E}_i : u_i+\widetilde{\gamma_{i}}\cdot u_i^2+\delta_i^{\prime}(v_i)+d_{i-1}^2+e_{i}^2=0.
 \end{equation}
Here $d_{i-1}$ (resp. $e_{i}$)  appears only when $L_{i-1}$ (resp. $L_{i+1}$) is \textit{of type I}.\\

We now choose a non-negative integer $j$ such that $L_j$ is \textit{of type I} and $L_{j+2}$ is \textit{of type II}.
For such a $j$, there is a non-negative integer $m_j$ such that $L_{j-2l}$ is \textit{of type I} for every $l$ with $0\leq l\leq m_j$ and
$L_{j-2(m_j+1)}$ is \textit{of type II}.
Then the sum of equations $$\sum_{l=0}^{m_j} \mathcal{F}_{j-2l}+\sum_{l}^{\prime}\widetilde{\gamma_{j+1-2l}}\cdot\mathcal{E}_{j+1-2l}$$ becomes
\begin{equation}
\sum_{l=0}^m (z_{j-2l}+z_{j-2l}^2)+\sum_{l}^{\prime}(\widetilde{\gamma_{j+1-2l}}\cdot u_{j+1-2l}+\widetilde{\gamma_{j+1-2l}}^2\cdot u_{j+1-2l}^2)=0.
\end{equation}
Here $\sum_{l}^{\prime}\widetilde{\gamma_{j+1-2l}}\cdot\mathcal{E}_{j+1-2l}$ is the sum of the equations
$\widetilde{\gamma_{j+1-2l}}\cdot\mathcal{E}_{j+1-2l}$'s such that $0 \leq l\leq m_j+1$ and $L_{j+1-2l}$ is \textit{of type} $\textit{I}^e$.
On the other hand, if $L_i$ is \textit{free} \textit{of type} $\textit{I}^e_1$ so that $v_i$ is $0$, the equation $\widetilde{\gamma_{i}}\cdot\mathcal{E}_i$ becomes
\begin{equation}
\widetilde{\gamma_{i}}\cdot u_i+\widetilde{\gamma_{i}}^2\cdot u_i^2=0.
\end{equation}

We next state remaining equations defining $\mathrm{Ker~}\varphi$.
In addition to the equations $\mathcal{F}_i$ and $\mathcal{E}_i$, Equation $(7.1)$ induces the following:
\begin{itemize}
\item  If $L_i$ is \textit{of type} $\textit{I}^o$, the $(1\times 2)$-block of   Equation $(7.1)$
 is
\begin{equation}
 \delta_i^{\prime}y_i+{}^tv_i+\mathcal{P}^i_{1, 2}=0.
\end{equation}
 Here, $\mathcal{P}^i_{1, 2}$ is a polynomial with variables $m_{i-1, i}, m_{i+1, i}$.
\item  If $L_i$ is \textit{of type} $\textit{I}^e$, the $(1\times 2)$-block of   Equation $(7.1)$
 is
\begin{equation}
\delta_i^{\prime}r_i+{}^tv_i=0.
\end{equation}
\item If $L_i$ is \textit{of type} $\textit{I}^e$, the $(1\times 3)$-block of  Equation $(7.1)$
is
\begin{equation}
  \delta_i^{\prime}t_i+{}^ty_i+{}^tv_iz_i+\widetilde{\gamma_i}\cdot{}^tv_i+\mathcal{P}^i_{1, 3}=0.
\end{equation}
 Here, $\mathcal{P}^i_{1, 3}$ is a polynomial with variables $m_{i-1, i}, m_{i+1, i}$.
\item If $L_i$ is \textit{of type} $\textit{I}^e$, the $(2\times 3)$-block of  Equation $(7.1)$ is
\begin{equation}
  x_i+w_i+{}^tr_i\delta_i^{\prime}t_i+z_i+u_iz_i+\widetilde{\gamma_i}\cdot u_i +\mathcal{P}^i_{2, 3}=0. 
\end{equation}
 Here, $\mathcal{P}^i_{2, 3}$ is a polynomial with variables $m_{i-1, i}, m_{i+1, i}$.\\
 \end{itemize}

Equation $(7.2)$ induces the following:
\begin{itemize}
\item If either $L_i$ or $L_{i+1}$ is \textit{of type II},  Equation $(7.2)$ becomes
\begin{equation}
{}^tm_{i,i}\delta_im_{i,i+1}+{}^tm_{i+1,i}\delta_{i+1}m_{i+1,i+1}=0.
\end{equation}
\item If both $L_i$ and $L_{i+1}$ are \textit{of type I},  Equation $(7.2)$ consists of two parts:
\begin{equation}
 \left\{
  \begin{array}{l l}
  {}^tm_{i,i}\delta_im_{i,i+1}+{}^tm_{i+1,i}\delta_{i+1}m_{i+1,i+1}=0   & \quad  \text{except for the $(n_i\times n_{i+1})$-entry};\\
  f_i+f_i^{\prime}+ \mathcal{P}_i^{\prime}=0 &   \quad  \text{for the $(n_i\times n_{i+1})$-entry}.
    \end{array} \right\}
\end{equation}
Here, $\mathcal{P}_i^{\prime}$ is a polynomial with variables
$m_{i,i}, m_{i,i+1}, m_{i+1,i}, m_{i+1,i+1}, m_{i-1,i}, m_{i-1,i+1},$ $m_{i+2,i}$ and $ m_{i+2,i+1}$.
Notice that the polynomial  $\mathcal{P}_i^{\prime}$ does not include the variables $f_i, f_i^{\prime}$.
We recall that $f_i$ (resp. $f_i^{\prime}$) is such that $2f_i$ (resp. $2f_i^{\prime}$) is the entry of $m_{i,i+1}$ (resp. $m_{i+1,i}$).\\
\end{itemize}

Finally, we observe Equation  (7.4), (7.5), (7.6), (7.7), (7.8), (7.9), (7.10), (7.11), (7.12), (7.13)
 and  (7.3).
The closed subscheme $\mathrm{Ker~}\varphi$ of $\mathrm{Ker~}\tilde{\varphi}$ is determined by these equations.
It is easily seen that  $\mathrm{Ker~}\varphi$ is a disconnected affine space with $2^{\alpha+\beta}$ components.
Moreover, the dimension of $\mathrm{Ker~}\varphi$ can be computed by observing these equations or by the following lemma, and it is $l$.
This completes the proof.
\end{proof}

The dimension of  $\mathrm{Ker~}\varphi$ is also computed from the following lemma easily.
\begin{Lem}
The dimension of $\mathrm{Ker~}\varphi$ is $l$.
\end{Lem}

\begin{proof}
Since the dimension of $\mathrm{Ker~}\varphi$ is at least $l$,
 it suffices to show that the dimension of the tangent space of $\mathrm{Ker~}\varphi$ at the identity $e$ is $l$.
It is enough to check the statement over the algebraic closure $\bar{\kappa}$ of $\kappa$.
Recall the morphism $$\tilde{\varphi} : \tilde{M} \rightarrow \prod_i\mathrm{Aut}_{\kappa}(\bar{V_i}).$$
Then we have $$\mathrm{Ker~}\varphi = \tilde{G} \cap \mathrm{Ker~}\tilde{\varphi}  $$ as closed subgroup schemes of $\tilde{M}$.
Recall that we have defined the map $\rho_{\ast, m}:T_m \rightarrow T_{\rho(m)}$ and
 we identified $T_m$ and $T_{\rho(m)}$ with $T_1(\bar{\kappa})$ and $T_2(\bar{\kappa})$ respectively, in the proof of Lemma 3.6.
Based on these, the tangent space of $\tilde{G}$ at $e$ is the kernel of the map $\rho_{\ast, e}$.
In addition, the tangent space of $\mathrm{Ker~}\tilde{\varphi}$ at $e$ is identified with the subspace of $T_1(\bar{\kappa})$,
satisfying with 9 cases described at the beginning of the Appendix
 if we change  $s_i=\mathrm{id}$ and $m_{i,i}=\mathrm{id}$ to  $s_i=0$ and $m_{i,i}=0$, respectively.
 We denote this subspace by $T_0(\bar{\kappa})$.

The tangent space of $\mathrm{Ker~}\varphi$ at $e$ is the intersection of $\mathrm{Ker~}\rho_{\ast, e}$ and  $T_0(\bar{\kappa})$ as subspaces of $T_1(\bar{\kappa})$.
 Thus it suffices to show that $\mathrm{Ker~}\rho_{\ast, e}\cap T_0(\bar{\kappa})$  has dimension $l$.

Since $X \mapsto q\cdot X $ is a bijection $T_1(\bar{\kappa}) \rightarrow T_3(\bar{\kappa})$ as explained in the proof of Lemma 3.6, it suffices to show that
$q\cdot(\mathrm{Ker~}\rho_{\ast, e} \cap T_0(\bar{\kappa}))$ as a subspace of $T_3(\bar{\kappa})$ has dimension $l$.
We have $$q\cdot(\mathrm{Ker~}\rho_{\ast, e}\cap T_0(\bar{\kappa}))=q\cdot \mathrm{Ker~}\rho_{\ast, e}\cap q\cdot T_0(\bar{\kappa}).$$
The space $q\cdot\mathrm{Ker~}\rho_{\ast, e}$ is the kernel of the map $T_3(\bar{\kappa}) \rightarrow T_2(\bar{\kappa}), Y \mapsto {}^t Y + Y$.
Thus it is the subspace of $T_3(\bar{\kappa})$ consisting of symmetric matrices $Y$'s whose diagonal entries
are 0.
Notice that this fact implies that the dimension of $q\cdot\mathrm{Ker~}\rho_{\ast, e}$ is $n(n-1)/2$, which is the dimension of $\tilde{G}$.
By considering $q\cdot T_0(\bar{\kappa})$, it is easily seen that the dimension of $q\cdot \mathrm{Ker~}\rho_{\ast, e}\cap q\cdot T_0(\bar{\kappa})$ is
exactly $l$.
\end{proof}


\end{document}